\documentclass{amsart}
\newtheorem{theorem}{Theorem}[section]

\usepackage[pagewise]{lineno}

\theoremstyle{definition}

\newtheorem{corollary}[theorem]{Corollary}
\newtheorem{definition}[theorem]{Definition}
\newtheorem{proposition}[theorem]{Proposition}
\newtheorem{remark}[theorem]{Remark}

\numberwithin{equation}{section}

\begin{document}

\title[Permanence for nonuniform hyperbolicity]{Permanence of nonuniform nonautonomous hyperbolicity for infinite-dimensional differential equations}


\author[T. Caraballo]{Tom\'as Caraballo$^1$}
\email{caraball@us.es}
\thanks{$^1$ Departamento de Ecuaciones Diferenciales y An\'alisis Num\'erico, Universidad de Sevilla, Spain.}
\author[A. N. Carvalho]{Alexandre N. Carvalho$^2$}
\thanks{$^2$ Instituto de Ci\^encias Matem\'aticas e de Computa\c c\~ao, Universidade de S\~ ao Paulo, Brazil.}
\email{andcarva@icmc.usp.br}

\author[J. A. Langa]{Jos\'e A. Langa$^1$}
\email{langa@us.es}


\author[A. N. Oliveira-Sousa]{Alexandre N. Oliveira-Sousa$^{2}$}
\email{alexandrenosousa@gmail.com}

\subjclass[2010]{Primary 37B55, 37B99, 34D09, 93D09}

\keywords{nonuniform exponential dichotomy, robustness, permanence of hyperbolic equilibria}

\date{}

\dedicatory{}

\begin{abstract}
	In this paper, we study stability properties of nonuniform hyperbolicity for evolution processes associated with differential equations in Banach spaces. 
	We prove a robustness result of nonuniform hyperbolicity for linear evolution processes, that is, we show that the property of admitting a nonuniform exponential dichotomy is stable under perturbation.
	Moreover, we provide conditions to obtain uniqueness and continuous dependence of projections
	associated with nonuniform exponential dichotomies. 
	We also present an example of evolution process in a Banach space that admits nonuniform exponential dichotomy and study the permanence of the nonuniform hyperbolicity under perturbation. Finally, we prove persistence of nonuniform hyperbolic solutions for nonlinear evolution processes under perturbations.
\end{abstract}

\maketitle

\section{Introduction}
\par In the framework of dynamical systems, hyperbolicity plays a fundamental role (see, e.g. \cite{Katok,C-Robinson,Shub} and the references therein). It is the key property for most of the results on permanence under perturbations. The permanence, on the other hand, is an essential property for dynamical systems that model real life phenomena. That importance is related to the fact that modelling always comes with approximations (due do the empiric nature that it caries) and/or with simplifications (introduced to make models treatable or simply because the complete set of variables that are related to the phenomenon is not known). Therefore, in order that the mathematical model reflects, in some way, the phenomenon modelled, it is essential that its dynamical structures are robust under perturbation. It starts with the robustness under perturbation of hyperbolicity itself. Here, we are concerned with the robustness of nonautonomous nonuniform hyperbolicity.

%
%

\par In the discrete case, \textit{hyperbolic dynamical systems}
$x_{n+1}=Bx_n$ 
appear when the spectrum of the
bounded linear operator $B$ does not intercept the unitary circle in the complex plane. This implies
the existence of a \textit{hyperbolic decomposition} of the space, which means 
that exist two main directions: one where the evolution of the dynamical system decays exponentially and another where it grows exponentially.
This property can be interpreted as a complete understanding of local or global dynamics. 
The set of operators that has such
decomposition is an open set in the spaces of bounded linear operators
and the operators in this set are called \textit{hyperbolic operators}. In other words, if $B$ is hyperbolic there is a neighborhood of $B$ such that every operator
in this neighborhood is hyperbolic. 
For autonomous differential equations, when $A$ is a bounded linear operator,
$\dot{x}=Ax$,
by the spectral mapping theorem \cite{Kato}, 
hyperbolicity
is associated with linear operators such that the spectrum does not intersect the imaginary line.
\par Generally, in nonautonomous differential equations, the notion of hyperbolicity is referred as \textit{exponential dichotomy}. More precisely,
consider the following differential equation in a Banach space $X$,
\begin{equation}\label{introdution-eq-standart-linear-equation}
\dot{x}=A(t)x, \ \ x(s)=x_s\in X.
\end{equation}
Under appropriate conditions, the solutions $x(t,s;x_s)$, $t\geq s$, of this initial value problem define an evolution process $\mathcal{S}:= \{S(t,s)\,; \, t\geq s \}$, where $S(t,s)x_s=x(t,s;x_s)$. 
We say that the evolution process  $\mathcal{S}$  admits an \textit{(uniform) exponential dichotomy} if there exists a family of projections,
$\{Q(t)\, ;\, t\in \mathbb{R}\}$ such that for each $t\geq s$ we have that
$S(t,s)Q(s)=Q(t)S(t,s)$, $S(t,s)$ is an isomorphism from $R(Q(s))$ onto $R(Q(t))$, and
\begin{eqnarray}\label{introdution-exp-dichotomy-def-eq1}
\|S(t,s)(Id_X-Q(s))\|_{\mathcal{L}(X)}&\leq &K e^{-\alpha(t-s)}, 
\ \ t\geq s;\\\label{introdution-exp-dichotomy-def-eq2}
\|S(t,s)Q(s)\|_{\mathcal{L}(X)}&\leq&  Ke^{\alpha(t-s)}, 
\ \ t< s,
\end{eqnarray}
for some constants $K\geq 1$ and $\alpha>0$.
Note that, since the 
vector field is changing in time, it is natural to think that for each initial time
we have a hyperbolic decomposition that resembles the properties in the autonomous case.
There is a long list of works through these last decades about existence of exponential dichotomies and their stability properties, for instance: \cite{Carvalho-Langa,Chow-Leiva-1,Chow-Leiva-2,Hale-Zhang,Henry-1,Henry-2,Pliss-Sell}.
\par If we replace the constant $K$ in the above definition by a continuous function $K(s)$ in \eqref{introdution-exp-dichotomy-def-eq1} and 
\eqref{introdution-exp-dichotomy-def-eq2},
we say that \eqref{introdution-eq-standart-linear-equation} admits a
\textit{nonuniform exponential dichotomy} (for an introduction see \cite{Barreira-Valls-Sta}).
Usually, the nonuniform bound is given by
$K(s)\leq De^{\nu|s|}$ for some $\nu>0$. 
If
$\alpha>\nu$, the properties of the nonuniform behavior resembles the uniform case, 
so this is the usual requirement to prove results about robustness and admissibility.
As in the uniform case,
there are many works concerning issues of existence and robustness for nonuniform exponential dichotomies
\cite{Alhalawa,Barreira-Dragicevi-Valls,BarreiraValls-Non,Barreira-Valls-Sta,Barreira-Valls-R,Barreira-Valls-Existence,Barreira-Valls-Robustness-noninvertible,Zhou-Lu-Zhang-1}. 
In this paper, we study robustness
and permanence of global hyperbolic solutions.

\par The robustness of nonuniform exponential dichotomy for equation
\eqref{introdution-eq-standart-linear-equation} can be interpreted as follows: 
suppose that the associated solution operator (evolution process) admits a nonuniform exponential dichotomy. The problem is to know
for which family of bounded linear operators $\{B(t): \, t\in \mathbb{R}\}$,
the perturbed problem
\begin{equation}\label{introdution-eq-standart-linear-equation-perturbed}
\dot{x}=A(t)x+B(t)x, 
\end{equation}	
admits a nonuniform exponential dichotomy.

\par Barreira and Valls
\cite{Barreira-Valls-R}
studied under which conditions the
nonuniform exponential dichotomy is robust in the case of 
invertible evolution processes. 
Later, Zhou \textit{et al.} \cite{Zhou-Lu-Zhang-1},
proved a similar result for random difference equations
for linear operators without the invertibility requirement. 
More recently,  Barreira and Valls
\cite{Barreira-Valls-Robustness-noninvertible}, 
proved that
nonuniform exponential dichotomy is robust
for continuous evolution processes, also without invertibility.
They consider an evolution process that admits a nonuniform exponential dichotomy with 
a general growth rate $\rho(\cdot)$. They proved that if $\alpha >2\nu$ and 
$B:\mathbb{R}\to \mathcal{L}(X)$ is continuous satisfying
$\|B(t)\|_{\mathcal{L}(X)}\leq \delta e^{-3\nu |\rho(t)| }\rho^\prime(t)$, 
for all $t\in \mathbb{R}$, then the perturbed problem 
\eqref{introdution-eq-standart-linear-equation-perturbed} admits a $\rho$-nonuniform exponential dichotomy.
\par We provide a interpretation of the robustness result as \textit{open property}. In fact,
if an evolution process $\mathcal{S}$ admits a nonuniform exponential dichotomy,
there is an open neighborhood $N(\mathcal{S})$ of $\mathcal{S}$ such that every evolution process in $N(\mathcal{S})$ also admits a
nonuniform exponential dichotomy. 
Our proof of the robustness result is inspired by the ideas of Henry \cite{Henry-1}. We prove that if a continuous evolution process admits a nonuniform exponential dichotomy, then each discretization also admits it. Then we use the \textit{roughness}
of the nonuniform exponential dichotomy for discrete evolution processes, obtained by Zhou \textit{et al.} \cite{Zhou-Lu-Zhang-1}, to obtain that each discretization of the perturbed evolution process also admits a nonuniform exponential dichotomy. Thus, to obtain our robustness result, we have to guarantee that if each discretization of a continuous evolution process $\mathcal{S}$ admits a nonuniform exponential dichotomy, then $\mathcal{S}$ also admits it. 
\par By this method, we obtain uniqueness and continuous dependence of projections, and explicit expressions for the bound and exponent of the perturbed evolution process. Besides,
since we preserve the condition $\alpha>\nu$ of Zhou \textit{et al.} \cite{Zhou-Lu-Zhang-1}, we obtain an improvement of the result of Barreira and Valls, at the particular case of Theorem 1 of
\cite{Barreira-Valls-Robustness-noninvertible}, when $\rho(t)=t$. 
Moreover, we do not assume that the evolution processes are invertible, then it is possible to apply our result on evolutionary differential equations in Banach spaces, as the ones that appears in \cite{Carvalho-Langa,Carvalho-Langa-Robison-book,Chow-Leiva-2,Henry-1}. 
\par An important consequence of the robustness result regarding nonlinear evolution processes is the persistence under perturbation of \textit{hyperbolic solutions}. More precisely,
consider a semilinear differential equation
\begin{equation}\label{introdution-eq-standart-semilinear-equation}
\dot{x}=A(t)x+f(t,x), \ \ x(s)=x_s\in X,
\end{equation}
and suppose that there for each $s\in \mathbb{R}$ and $x_s\in X$ there exist a solution 
$x(\cdot,s;x_s):[s,+\infty)\to X$, then there exist a nonlinear evolution process 
$\mathcal{S}_f=\{S_f(t,s):t\geq s\}$ defined by $S_f(t,s)x=x(t,s;x_s)$. A map $\xi:\mathbb{R}\to X$ is called a \textit{global solution} for 
$\mathcal{S}_f$ if $S_f(t,s)\xi(s)=\xi(t)$ for every $t\geq s$ and we say that $\xi$ is a \textit{nonuniform hyperbolic solution} if the linearized evolution process over $\xi$ admits a nonuniform exponential dichotomy. This notion also appears in Barreira and Valls \cite{Barreira-Valls-Sta} as \textit{nonuniformly hyperbolic solutions}. 
In the uniform case, in Carvalho and Langa 
\cite{Carvalho-Langa}, it was studied the existence of hyperbolic solutions from nonautonomous perturbation of hyperbolic equilibria for a autonomous semilinear differential equation. 
\par Inspired by Carvalho \textit{et al.} \cite{Carvalho-Langa-Robison-book} we prove a result on the persistence of nonuniform hyperbolic solutions under perturbations. In fact, if $\xi$ is a nonuniform hyperbolic solution for $\mathcal{S}_f$ and $g$ is a map ``close" to $f$, then there exists a nonuniform hyperbolic solution for $\mathcal{S}_g$ ``close" to $\xi$. Additionally, we also prove that bounded nonuniform hyperbolic are \textit{isolated} in the space of bounded continuous functions $C_b(\mathbb{R},X)$, i.e., if $\xi$ is a nonuniform 
hyperbolic solution, then there exists a neighborhood of $\xi$ in $C_b(\mathbb{R})$ such that $\xi$ is the only bounded solution for $\mathcal{S}_f$ is this neighborhood. 
\par Therefore, the aim of this work is to establish: uniqueness and continuity for family of projections
associated with nonuniform exponential dichotomy; a robustness result with the condition $\alpha>\nu$; and persistence of nonuniform hyperbolic solutions. To that end, in Section \ref{section-robustness-discrete-case}, we summarize some important facts for discrete evolution processes with nonuniform exponential dichotomy. We prove uniqueness and continuity of projections and briefly recall the robustness result of Zhou \textit{et al.} \cite{Zhou-Lu-Zhang-1} in our framework.
Then, in Section \ref{section-robustness-continuous-case}, we prove 
uniqueness and continuous dependence for family of projections, and a robustness result of nonuniform exponential dichotomy for continuous evolution processes.
%
In Section \ref{subsection-a-general-example}, we present a class of examples of evolutions processes in a Banach space that admit nonuniform exponential dichotomy.
Finally, in Section \ref{section-persistence}, 
we consider
\textit{nonuniform hyperbolic solutions} for evolution processes associated with semilinear differential equations. We prove that these solutions are isolated in $C_b(\mathbb{R},X)$ and that they persist under perturbations.

\section{Nonuniform exponential dichotomy: discrete case}\label{section-robustness-discrete-case}

\par In this section, we present some basic facts of nonuniform exponential dichotomy for discrete evolution processes. We briefly recall some results of Zhou \textit{et al.} \cite{Zhou-Lu-Zhang-1} (without considering a random parameter) that we will use to prove our results for the continuous case. Their most important result is a robustness theorem for the nonuniform exponential dichotomy. Moreover, as a consequence of the results in \cite{Zhou-Lu-Zhang-1} we obtain uniqueness and continuous dependence of projections associated with the nonuniform exponential dichotomy.
%
%
\par We start with the definition of a \textit{discrete evolution process} in a Banach space 
$(X,\|\cdot\|_X)$ in a particular case where the family of operators are linear bounded operators in $X$.
\begin{definition}
	Let 
	$\mathcal{S}=\{S_{n,m}\, : \, n\geq m \hbox{ with }
	n,m\in \mathbb{Z}\}$ 
	be a family
	of bounded linear operators in a Banach space $X$. 
	We say that
	$\mathcal{S}$ is a \textbf{discrete evolution process} if
	\begin{enumerate}
		\item 
		$S_{n,n}=Id_X$, for all $n\in \mathbb{Z}$;
		\item 
		$S_{n,m}S_{m,k}=S_{n,k}$, for all $n\geq m\geq k$.
	\end{enumerate}
	To simplify the notation, we only write $\mathcal{S}=\{S_{n,m}\, :  n\geq m \}$ as an evolution process, whenever is clear are dealing with discrete ones.
\end{definition}

\begin{remark}
	It is always possible to associate a discrete evolution process
	$\mathcal{S}=\{S_{n,m}\, : \, n\geq m \}$ 
	with the family index with one parameter, i.e.,
	$S_n=S_{n+1,n}$ for all $n\in \mathbb{Z}$.
	Conversely, if we have a family 
	$\{S_n\, :\, n\in \mathbb{Z}\}\subset \mathcal{L}(X)$ 
	we define
	$S_{n,m}:=S_{n-1}\cdots S_m$ for $n>m$ and
	$S_{n,n}:=Id_X$
	and obtain a linear evolution process $\mathcal{S}=\{S_{n,m}\, : \, n\geq m \}$.
	Therefore, when we refer to a discrete evolution process
	$\mathcal{S}$ we can use both notations.
\end{remark}

\par The definition above is natural in the study of \textit{difference equations}.
Given a discrete linear evolution process $\mathcal{S}=\{S_n:n\in \mathbb{Z}\}$ it is possible 
to consider the following difference equation 
\begin{equation}\label{discrete-equation-associated-with-the-process}
x_{n+1}=S_n x_n, \ \ x_n\in X, \ \ n\in \mathbb{Z}.
\end{equation}

Now, we present the definition of \textit{nonuniform exponential dichotomy}.

\begin{definition}
	Let 
	$\mathcal{S}=\{S_{n,m} \, : \, n\geq m\}\subset \mathcal{L}(X)$
	be an evolution process in a Banach space $X$.
	We say that $S$ admits a \textbf{nonuniform exponential dichotomy} if there is a family of continuous projections
	$\{Q_n; \, n\in \mathbb{Z}\}$ in $\mathcal{L}(X)$ such that
	\begin{enumerate}
		\item $Q_nS_{n,m}=S_{n,m}Q_m$, for $n\geq m$;
		\item $S_{n,m}:R(Q_m ) \to R(Q_n )$ is an isomorphism, 
		for $n\geq m$, and
		we define $S_{m,n}$ as its inverse;
		\item There exists a function $K:\mathbb{Z}\rightarrow [1,+\infty)$ with
		$K(n)\leq De^{\nu |n|}$, for some $D\geq 1$ and $\nu>0$,
		and $\alpha>0$ such that
		\begin{equation*}
		\|S_{n,m}(Id_X-Q_m )\|_{\mathcal{L}(X)}\leq K(m) e^{-\alpha(n-m)}, \ \ 
		\forall n\geq m;
		\end{equation*}
		and
		\begin{equation*}
		\|S_{n,m}Q_m\|_{\mathcal{L}(X)}\leq K(m) e^{\alpha(n-m)}, \ \ 
		\forall n\leq m.
		\end{equation*}
	\end{enumerate}
\end{definition}

\par In this theory, $K$ and $\alpha$ are usually called the \textbf{bound} and the \textbf{exponent} of the exponential dichotomy, respectively.


\par We now recall the definition of a \textit{Green function}.
\begin{definition}
	Let $\mathcal{S}=\{S_n:n\in \mathbb{Z}\}$ be a discrete evolution process such that admits 
	a nonuniform exponential dichotomy with family of projections 
	$\{Q_n\}_{n\in \mathbb{Z}}$. 
	The \textbf{Green function} associated to the evolution process $\mathcal{S}$ is given by
	\begin{equation*}
	G_{n,m}= \left\{ 
	\begin{array}{l l} 
	S_{n,m}(Id_X-Q_m), 
	&  \quad \hbox{if } n\geq m,
	\\ -S_{n,m}Q_m, \, 
	& \quad \hbox{if } n<m.
	\end{array} 
	\right.
	\end{equation*} 
	
\end{definition}

\par A space that appears naturally when dealing with
nonuniform exponential dichotomies is
\begin{equation*}
l_{1/K}^\infty(\mathbb{Z}):=\big\{f:\mathbb{Z}\to X: \,
\sup_{n\in \mathbb{Z}} \big\{\|f_n\|_X K(n+1)\big\}=M_f <+\infty \big\},
\end{equation*} 
where $K:\mathbb{Z}\to \mathbb{R}$ is such that
$K(n)\geq 1$ for all $n\in \mathbb{Z}$.

\par As in the uniform case, the next result shows that it is possible to obtain the solution for
\begin{equation}\label{discrete-equation-associated-with-the-process-perturbed}
x_{n+1}=S_n x_n+f_n.
\end{equation} 
using the \textit{Green function}.

\begin{theorem}\label{th-admissibility-pair-discrete-case}
	Assume that
	the evolution process $\mathcal{S}=\{S_n:n\in \mathbb{Z}\}$
	admits a nonuniform exponential dichotomy
	with bound $K(n)\leq De^{\nu|n|}$ and exponent $\alpha>\nu$. If
	$f\in l_{1/K}^\infty(\mathbb{Z})$,
	then
	\eqref{discrete-equation-associated-with-the-process-perturbed} 
	possesses a bounded solution given by
	\begin{equation*}
	x_n=\sum_{-\infty}^{+\infty}G_{n,k+1}f_k, \ \ \forall \  n\in \mathbb{Z}.
	\end{equation*}
\end{theorem}
\par For the proof of Theorem \eqref{th-admissibility-pair-discrete-case} see Zhou \textit{et al.} \cite{Zhou-Lu-Zhang-1}. Note that, in \cite{Zhou-Lu-Zhang-1} they consider tempered exponential dichotomies, but their proof hold true with the condition $\alpha>\nu$. 

\par As a consequence of Theorem \ref{th-admissibility-pair-discrete-case}, we obtain uniqueness of the family of projections associated with the nonuniform exponential dichotomy.
\begin{corollary}[Uniqueness of projections]
	\label{cor-uniqueness-projection-discrete}
	If $\mathcal{S}=\{S_n:n\in \mathbb{Z}\}$ admits a nonuniform exponential dichotomy with bound $K(n)\leq De^{\nu|n|}$ and exponent $\alpha>\nu$, then
	the family of projections are uniquely determined.
\end{corollary}
\begin{proof}
	Let $\{Q_n^{(i)}\, ; \, n\in \mathbb{Z}\}$, for $i=1,2$, 
	projections associated with the evolution process $\mathcal{S}$.
	Given $x\in X$ and $m\in \mathbb{Z}$ fixed, define $f_n=0$, for all $n\neq m-1$, and $f_{m-1}=K(m)^{-1}x$.
	Thus, $f\in l_{1/K}^\infty(\mathbb{Z})$ and from Theorem \ref{th-admissibility-pair-discrete-case} there exists a unique solution $\{x_n\}_{n\in \mathbb{Z}}$ for 
	\begin{equation*}
	x_{n+1}=S_n x_n +f_n, \ \ n\in \mathbb{Z}.
	\end{equation*}
	Hence,
	$x_m=\sum_{-\infty}^{+\infty}G_{m,k+1}^{(i)}f_k=G_{m,m}^{(i)}f_{m-1}=K(m)^{-1}(Id_X-Q_m^{(i)})x$,
	for $i=1,2$.
	Therefore, $Q_m^{(1)}=Q_m^{(2)}$.
\end{proof}

\par Next, we establish a result on continuous dependence of projections.
\begin{theorem}[Continuous dependence of projections]
	\label{th-continuity-projection}
	Suppose that $\{T_n\}_{n\in {\mathbb Z}}$ and
	$\{S_n\}_{n\in {\mathbb Z}}$
	admits a nonuniform exponential dichotomy with projections 
	$\{Q_n^{\mathcal{T}}\}_{n\in {\mathbb Z}}$ and 
	$\{Q_n^{\mathcal{S}}\}_{n\in {\mathbb Z}}$,
	exponents $\alpha_\mathcal{T}$
	and 
	$\alpha_\mathcal{S}$, respectively,
	and the same bound
	$K(n)\leq De^{\nu|n|}$.
	If $\nu<\min\{\alpha_\mathcal{T},\alpha_\mathcal{S}\}$ and 
	\begin{equation*}
	\sup_{n\in \mathbb{Z}} \big\{K(n+1)\|T_n-S_n\|_{{\mathcal L}(X)}\big\}\leq
	\epsilon,
	\end{equation*}
	then
	$$
	\sup_{n\in \mathbb{Z}} \big\{K(n)^{-1} \, \|Q_n^{\mathcal{T}}-Q_n^{\mathcal{S}}\|_{{\mathcal L}(X)}\big\}
	\leq\frac{e^{-\alpha_\mathcal{S}} + e^{-\alpha_\mathcal{T}}}{1-e^{-(\alpha_\mathcal{S}+\alpha_\mathcal{T})}}\, \epsilon .
	$$
\end{theorem}

\begin{proof}
	Let $z\in X$ and $m\in \mathbb{Z}$ be fixed and consider
	\begin{equation*}
	f_n= \left\{ 
	\begin{array}{l l} 
	0,
	& \quad \hbox{if } n\neq m-1,
	\\ K(m)^{-1}z, \, 
	& \quad \hbox{if } n=m-1.
	\end{array} 
	\right.
	\end{equation*} 
	Thus, by Theorem \ref{th-admissibility-pair-discrete-case},
	there exist bounded solutions $x^k=\{x_n^k\}_{n\in\mathbb{Z}}$ given by
	$x_n^{k}:=G_{n,m}^{k} zK(m)^{-1}$ for $k=\mathcal{T}, \mathcal{S}$. 
	Note that, for $n\in \mathbb{Z}$,
	\begin{equation*}
	x_{n+1}^{\mathcal{T}} - S_n x_n^{\mathcal{T}}=
	T_n x_n^{\mathcal{T}}-S_n x_n^{\mathcal{T}}+f_n  
	\end{equation*}
	and
	$	x_{n+1}^{\mathcal{S}} - S_n x_n^{\mathcal{S}}=f_n$.
	Then, if
	$z_n:=x_n^{\mathcal{T}}-x_n^{\mathcal{S}}$ we obtain that  
	$z_{n+1}=S_n z_n + y_n$,
	where
	$y_n=(T_n-S_n)x_n^{\mathcal{T}}$ for all $n\in \mathbb{Z}$. 
	Thanks to the boundedness of the sequence $\{x_n^{\mathcal{T}}\}_{n\in\mathbb{Z}}$ 
	and by the hypothesis on $T_n-S_n$ we have that 
	$\{y_n K(n+1)\}_{n\in {\mathbb Z}}$ is bounded, and by 
	Theorem \ref{th-admissibility-pair-discrete-case} 
	we have that
	\begin{equation*}
	z_n=\sum_{k=-\infty}^\infty G_{n,k+1}^{\mathcal{S}}
	(T_k-S_k)G_{k,m}^{\mathcal{T}} zK(m)^{-1},
	\end{equation*}
	and therefore, by the hypothesis on 
	$\mathcal{T}-\mathcal{S}$,
	we deduce
	\begin{eqnarray*}
		\|z_m\|_X&\leq&\sum_{k=-\infty}^\infty K(k+1) e^{-\alpha_\mathcal{S}|m-k-1|}
		\|T_k-S_k\|_{\mathcal{L}(X)}e^ {-\alpha_\mathcal{S}|k-m|}\|z\|_{X} \\
		&\leq& \frac{e^{-\alpha_\mathcal{S}} + e^{-\alpha_\mathcal{T}}}{1-e^{-(\alpha_\mathcal{S}+\alpha_\mathcal{T})}}\, \epsilon \,
		\|z\|_X.
	\end{eqnarray*}
	The definition of $z$ in $m$ yields
	$$
	z_m=x_m^{\mathcal{T}}-x_m^{\mathcal{S}}=(G_{m,m}^{\mathcal{T}}-G_{m,m}^{\mathcal{S}})K(m)^{-1}z=(Q_m^{\mathcal{S}}-Q_m^{\mathcal{T}})K(m)^{-1}z.
	$$
	Consequently,
	\begin{equation*}
	\|(Q_m^{\mathcal{S}}-Q_m^{\mathcal{T}})K(m)^{-1}z\|_X\leq \frac{e^{-\alpha_\mathcal{S}} + e^{-\alpha_\mathcal{T}}}{1-e^{-(\alpha_\mathcal{S}+\alpha_\mathcal{T})}}\, \epsilon \,
	\|z\|_X,
	\end{equation*}
	which concludes the proof of the theorem.
\end{proof}

\par Finally, we state a robustness result for discrete evolution processes with nonuniform exponential dichotomies.
\begin{theorem}[Robustness for discrete evolution processes]
	\label{th-roughness-discrete-TED}
	Let $\mathcal{S}=\{S_n:n\in \mathbb{Z}\}$, $\mathcal{B}=\{B_n:n\in \mathbb{Z}\} \subset \mathcal{L}(X)$ be discrete evolution processes. 
	Assume that $\mathcal{S}$ admits
	a nonuniform exponential dichotomy with bound $K(n)\leq De^{\nu|n|}$ and exponent $\alpha>\nu$, and that $\mathcal{B}$ satisfies
	\begin{equation*}
	\|B_k\|_{\mathcal{L}(X)} \leq \delta K(k+1)^{-1},  \ \forall k\,\in \mathbb{Z},
	\end{equation*}
	where $\delta>0$ is such that $\delta<(1-e^{-\alpha})/(1-e^{-\alpha})$.
	Then the perturbed evolution process 
	$\mathcal{T}=\mathcal{S}+\mathcal{B}$
	admits a nonuniform exponential dichotomy with exponent
	\begin{equation*}
	\tilde{\alpha}=-\ln(\cosh \alpha - [\cosh ^2 \alpha-1-2\delta\sinh \alpha]^{1/2}),
	\end{equation*}
	and  bound
	\begin{equation*}
	\tilde{K}(n)=K(n)\left[1+\frac{\delta}{(1-\rho)(1-e^{-\alpha})}\right]\max{[D_1,D_2]},
	\end{equation*}
	where $\rho:=\delta(1+e^{-\alpha})/(1-e^{-\alpha})$,
	$D_1:=[1-\delta e^{-\alpha}/(1-e^{-\alpha-\tilde{\alpha}})]^{-1}$,
	$D_2:=[1-\delta e^{-\tilde{\beta}}/(1-e^{-\alpha-\tilde{\beta}})]^{-1}$
	and
	$\tilde{\beta}:=\tilde{\alpha}+\ln(1+2\delta\sinh\alpha)$.
\end{theorem}
\par The proof of Theorem \ref{th-roughness-discrete-TED} follows by the same arguments of the proof of Theorem 1 of \cite{Zhou-Lu-Zhang-1} with minimal changes. 
It is important to notice that all the arguments of their proof still hold with the assumption $\alpha>\nu$. We reinforce, that one of our goals is to prove a robustness result of nonuniform exponential for \textit{continuous evolution processes} with this same condition on the exponents ($\alpha>\nu$). 

\section{Nonuniform exponential dichotomy: continuous case}
\label{section-robustness-continuous-case}

In this section, we consider evolution processes with parameters in $\mathbb{R}$.
Inspired by the ideas of Henry \cite{Henry-1}, we prove theorems that allows us to obtain the continuous versions of the results presented in Section \ref{section-robustness-discrete-case}.
The main theorem of this section is our robustness result for nonuniform exponential dichotomies, namely Theorem \ref{th-roughness-continuous-TED}, and we also provide a version of it to be applied in differential equations, Theorem \ref{th-perturation-of-nununiform-exp-dcihotomy-for-differential-eq}. In addition, we establish results on the uniqueness and continuous dependence of projections  associated with nonuniform exponential dichotomy, Corollary \ref{cor-uniqueness-projection-continuous} and Theorem \ref{th-continuous-depende-projections}, respectively. 


%
%
%

\par We define a \textit{continuous evolution process} in $X$ as follows.

\begin{definition}
	Let $\mathcal{S}:=\{S(t,s):X\to X\, ; \, t\geq s, \  t,s\in \mathbb{R}\}$ be a family of continuous operators in a Banach space $X$.
	We say that $\mathcal{S}$ is a \textbf{continuous evolution process} in $X$ if
	\begin{enumerate}
		\item $S(t,t)=Id_X$, for all $t\in \mathbb{R}$;
		\item $S(t,s)S(s,\tau)=S(t,\tau)$, for $t\geq s\geq \tau$;
		\item $\{(t,s)\in \mathbb{R}^2 ; \, t\geq s\}\times X\ni (t,s,x)\mapsto S(t,s)x$ is continuous.
	\end{enumerate}
\par To simplify we usually say that $\mathcal{S}=\{S(t,s):t\geq s\}$ is an \textbf{evolution process}, whenever is implicit that $\mathcal{S}$ is a continuous evolution process.
\end{definition}

\begin{remark}
	Note that the operators $S(t,s):X\to X$, in the definition above, do not need to be linear. In fact, in Section \ref{section-persistence}, we study permanence of the nonuniform hyperbolic behavior for nonlinear evolution processes. 
\end{remark}
\par We also recall the notion of a \textit{global solution} for an evolution process.
\begin{definition}\label{def-global-solution}
	Let $\mathcal{S}=\{S(t,s):t\geq s\}$ be an evolution process. We say that $\xi:\mathbb{R}\to X$ is a 
	\textbf{global solution} for $\mathcal{S}$ if $S(t,s)\xi(s)=\xi(t)$ for every $t\geq s$.
	\par We say that a global solution $\xi$ is \textbf{backwards bounded} if there exists 
	$t_0\in \mathbb{R}$ such that $\xi(-\infty, t_0]=\{\xi(t): t\leq t_0\}$ is bounded.
\end{definition}

\par Now, we present the definition of \textit{nonuniform exponential dichotomy} for linear evolution processes:
\begin{definition}
	Let $\mathcal{S}=\{S(t,s) \, ; \, t\geq s\}\subset \mathcal{L}(X)$ be an evolution process. 
	We say that $\mathcal{S}$ admits a
	\textbf{nonuniform exponential dichotomy} if there exists 
	a family of continuous projections $\{Q(t)\, : \, t\in \mathbb{R}\}$ such that
	\begin{enumerate}
		\item $Q(t) S(t,s)= S(t,s) Q(s)$, for all $t\geq s$;
		\item $S(t,s):R( Q(s) ) \to R( Q(t) )$ is an isomorphism, for $t\geq s$, and we define 
		$S(s,t)$ as its inverse;
		\item There exists a continuous function
		$K:\mathbb{R}\rightarrow [1,+\infty)$ and some constants $\alpha>0$, 
		$D\geq 1$ and $\nu\geq 0$ such that
		$K(s)\leq De^{\nu|s|}$ and
		\begin{eqnarray*}
			\|S(t,s)(Id_X-Q(s))\|_{\mathcal{L}(X)}&\leq& K(s) e^{-\alpha(t-s)}, 
			\ \ t\geq s;\\
			\|S(t,s)Q(s)\|_{\mathcal{L}(X)}&\leq& K(s) e^{\alpha(t-s)}, 
			\ \ t< s.
		\end{eqnarray*}
	\end{enumerate}
\end{definition}

\begin{remark}
	This definition also includes \textit{uniform exponential dichotomies}, when $K$ is bounded, and \textit{tempered
		exponential dichotomies}, when $t\mapsto K(t)$ has a sub-exponential growth, see \cite{Barreira-Dragicevi-Valls,Zhou-Lu-Zhang-1}.
\end{remark}

\par In the following result we study each ``discretization at instant $t$" of an evolution process that admits a nonuniform exponential dichotomy. 

\begin{theorem}\label{th-continuousTED-implies-discreteTED}
	Let $\mathcal{S}$ be a continuous evolution process that admits a nonuniform exponential dichotomy with bound $K(t)=De^{\nu |t|}$ and exponent $\alpha>0$. 
	Then for each $t\in \mathbb{R}$ and $l>0$ the discrete evolution process
	\begin{equation*}
	\{S_{m,n}(t)\, :\,m,n\in \mathbb{Z}\, \hbox{with  } m\geq n \}:=\{S(t+ml,t+nl)\, :\,m,n\in \mathbb{Z}\,\hbox{with  } m\geq n \}
	\end{equation*}
	admits a nonuniform exponential dichotomy with bound $\tilde{K}_t(m):=K(t+ml)$ and exponent $\tilde{\alpha}=\alpha l$.
\end{theorem}

\begin{proof}
	Define, for each $t\in \mathbb{R}$, the family of projections
	$\{Q_m(t)=Q(t+ml)\, : \, m\in \mathbb{N}\}$, then
	\begin{eqnarray*}
		Q_m(t)S_{m, n}(t)&=&Q(t+ml)S(t+ml,t+nl)\\
		&=&S(t+ml,t+nl)Q(t+nl)\\
		&=&S_{m, n}(t)Q_n(t),
	\end{eqnarray*}
	and the first property is proved.
	Note that, for $m\geq n$,
	\begin{equation*}
	S_{m,n}(t)|_{R(Q_n(t)  )}
	=S(t+ml,t+nl)|_{R(Q(t+nl)  )}
	\end{equation*}
	and the right hand side of the equation is an isomorphism, so we define the inverse
	$S_{n,m}(t):R(Q(t+ml))\to R(Q(t+nl))$.
	\par Finally, for $n\geq m$,
	\begin{eqnarray*}
		\|S_{n,m}(t)(Id_X-Q_m(t))\|_{\mathcal{L}(X)} &=& \|S(t+ml,t+nl)(Id_X-Q(t+nl))\|_{\mathcal{L}(X)}\\
		&\leq& K(t+ml) e^{-\alpha l(n-m)},
	\end{eqnarray*}
	and, for $n<m$,
	\begin{eqnarray*}
		\|S_{n,m}(t)Q_m(t)\|_{\mathcal{L}(X)} &=&\|S_{n,m}(t)Q(t+ml)\|_{\mathcal{L}(X)} \\
		&\leq& K(t+ml) e^{\alpha l(n-m)}.
	\end{eqnarray*}
	Therefore, $\{S_{n,m}(t):n\geq m\}$ admits a discrete nonuniform exponential dichotomy with exponent $\tilde{\alpha}=\alpha l$ and bound
	$\tilde{K}_t(m)=K(t+ml)\leq De^{\nu |t|} e^{\nu l |m|}$, which concludes the proof.
\end{proof}

\begin{remark}
	In Theorem \ref{th-continuousTED-implies-discreteTED}, for a fixed 
	$t\in \mathbb{R}$,
	the discretized evolution process 
	$\{S_n(t)\, : \, n\in \mathbb{Z}\}$ possesses with a bound $K_t$ dependent of the time $t$ and the
	exponent $\tilde{\alpha}$ is independent of $t$. This is an expected difference with the the case of uniform exponential dichotomy, where both, the bound and the exponent of the discretization are independent of $t$, see Henry \cite{Henry-1}.
\end{remark}

\par Now, as a consequence of Theorem \ref{th-continuousTED-implies-discreteTED}
and Corollary \ref{cor-uniqueness-projection-discrete}, we obtain the uniqueness of 
the family of projections. 

\begin{corollary}[Uniqueness of the family of projections]
	\label{cor-uniqueness-projection-continuous}
	Let $\mathcal{S}$ be an evolution process such that admits a nonuniform exponential dichotomy with bound $K(t)\leq De^{\nu|t|}$ and exponent $\alpha>\nu$. Then the family of projections
	is unique.
\end{corollary}

\par As another application of Theorem
\ref{th-continuousTED-implies-discreteTED},
we prove a result on the continuous dependence of projections.

\begin{theorem}[Continuous dependence of projections]
	\label{th-continuous-depende-projections}
	Suppose that $\mathcal{S}$ and $\mathcal{T}$ are linear evolution processes
	with nonuniform exponential dichotomy with projections 
	$\{Q^\mathcal{S}(t):t\in \mathbb{R}\}$ and
	$\{Q^\mathcal{T}(t):t\in \mathbb{R}\}$ and exponents 
	$\alpha_\mathcal{T},\alpha_\mathcal{S}$ and with the same bound
	$K$. If $\nu<\min\{\alpha_\mathcal{T},\alpha_\mathcal{S}\}$ and
	\begin{equation}\label{eq-th-continuous-depende-projections}
	\sup_{0\leq t-s\leq 1}\big\{K(t)\|T(t,s)-S(t,s)\|_{\mathcal{L}(X)}\big\}\leq \epsilon,
	\end{equation}
	then
	\begin{equation*}
	\sup_{t\in \mathbb{R}}\big\{K(t)^{-1}\|Q^\mathcal{T}(t)-Q^\mathcal{S}(t)\|_{\mathcal{L}(X)}\big\}
	\leq \frac{e^{-\alpha_\mathcal{S}} + e^{-\alpha_\mathcal{T}}}{1-e^{-(\alpha_\mathcal{S}+\alpha_\mathcal{T})}} \epsilon.
	\end{equation*}
\end{theorem}

\begin{proof}
	From Theorem \ref{th-continuousTED-implies-discreteTED},
	for each $t_0\in\mathbb{R}$ and $0<l\leq 1$ we have that
	both
	$\{T_n(t_0):n\in \mathbb{Z}\}$ and
	$\{S_n(t_0):n\in \mathbb{Z}\}$ admit a nonuniform exponential dichotomy with 
	exponents
	$\alpha_\mathcal{T} l$ and $\alpha_\mathcal{S} l$
	and the same bound
	$K_{t_0}(n):=K(t_0+nl)$.
	Now, from Theorem \ref{th-continuity-projection} we conclude that
	\begin{equation*}
	K(t_0+nl)^{-1}\|Q^\mathcal{T}(t_0+nl)-Q^\mathcal{S}(t_0+nl)\|_{\mathcal{L}(X)}
	\leq \frac{e^{-\alpha_\mathcal{S}} + e^{-\alpha_\mathcal{T}}}{1-e^{-(\alpha_\mathcal{S}+\alpha_\mathcal{T})}} \epsilon.
	\end{equation*}
	To conclude the proof note that for any $t\in\mathbb{R}$ it is 
	possible to find $t_0$ and $0<l\leq 1$ such that
	$t=t_0+nl$.
\end{proof}

\par Uniqueness and continuous dependence of projections were a simple consequence of Theorem
\ref{th-continuousTED-implies-discreteTED}, and of course the results at the discrete case. However, to prove our robustness result, we will need a sort of a reciprocal result of Theorem 
\ref{th-continuousTED-implies-discreteTED}.

\begin{theorem}\label{th-discrete-dichotomy-implies-continuous-dichotomy}
	Let $\mathcal{S}:\{S(t,s):t\geq s\}\subset\mathcal{L}(X)$ 
	be a continuous evolution process.
	Suppose that
	\begin{enumerate}
		\item there exist $l>0$ and $\nu \geq 0$ such that
		\begin{equation*}
		L(\nu,l):=\sup_{0\leq t-s\leq l}\big\{ \|S(t,s)\|_{\mathcal{L}(X)} \,e^{-\nu |t|} \big\} < +\infty,
		\end{equation*}
		\item for each $t\in \mathbb{R}$ the discretized process,
		\begin{equation*}
		\{T_{n,m}(t) ,\, n\geq m\}=\{S(t+nl,t+ml),\, n\geq m\}
		\end{equation*}
		possesses a nonuniform exponential dichotomy with bound 
		$K_t(\cdot):\mathbb{Z}\rightarrow [1,+\infty)$, 
		with
		$K_t(m)\leq De^{\nu |t+m|}$ and exponent $\alpha>0$ independent of $t$.
	\end{enumerate}
	If $\nu \, l<\alpha $, the evolution process $\mathcal{S}$ admits a nonuniform exponential dichotomy
	with exponent $\hat{\alpha}=(\alpha-\nu l)/l$ and bound 
	\begin{equation*}
	\hat{K}(s)=
	D^2 e^{2\alpha} \max\{L(\nu,l), L(\nu,l)^2\}
	\,	e^{2\nu|s|}.
	\end{equation*}
\end{theorem}

\begin{proof}
	First, we fix $t\in \mathbb{R}$ and define the linear operator 
	$T_n(t):=T_{n+1,n}(t)$, for each $n\in \mathbb{Z}$. Then for each discrete evolution process
	$\{T_n(t)\, : \, n\in \mathbb{Z}\}$, there exists a family of projections
	$\{Q_n(t)\, : \, n\in \mathbb{Z}\}$ such that
	the nonuniform exponential dichotomy conditions are satisfied. 
	\par For each fixed $k\in \mathbb{Z}$ we have
	\begin{equation*}
	T_{n+k}(t) = T_n(t+kl), \ \ \forall n \in \mathbb{Z}.
	\end{equation*}
	Then this linear operator generates the same evolution process 
	with associated projections
	$\{Q_{n+k}(t)\}_{n\in\mathbb{Z}}$ and $\{Q_{n}(t+kl)\}_{n\in\mathbb{Z}}$. Thus by uniqueness of 
	the projections for the 
	discrete case, namely
	Corollary \ref{cor-uniqueness-projection-discrete},
	we obtain that for all $n,k\in \mathbb{Z}$,
	\begin{equation*}
	Q_{n+k}(t)=Q_n(t+kl).
	\end{equation*}
	Now, for all $t\in \mathbb{R}$ we define $Q(t):=Q_0(t)$.
	These projections are the candidates to obtain the nonuniform exponential dichotomy. 
	\par Let us now prove the boundedness in the case $t\geq s$.
	\par \textbf{Claim 1:}
	If $t\geq s$, then
	\begin{equation*}
	\|S(t,s)(Id_X-Q(s))\|_{\mathcal{L}(X)} \leq \hat{K}(s) 
	e^{-\hat{\alpha}(t-s)},
	\end{equation*}	
	where $\hat{K}$ is defined in the statement of the theorem.
	\par In fact,
	choose $n\in \mathbb{N}$, such that 
	$nl+s\leq t<(n+1)l+s$, then
	we write 
	\begin{equation*}
	S(t,s)(Id_X-Q(s))=
	S(t,s+nl) S(s+nl,s)(Id_X-Q_0(s)).
	\end{equation*}	
	Thus, by hypothesis,
	\begin{equation*}
	\|S(s+nl,s)(Id_X-Q_0(s))\|_{\mathcal{L}(X)}=
	\|T_{n,0}(s)(Id_X-Q_0(s))\|_{\mathcal{L}(X)}
	\leq K_s(0)e^{-\alpha n},
	\end{equation*}
	which implies that
	\begin{eqnarray*}
		\|S(t,s)(Id_X-Q(s))\|_{\mathcal{L}(X)}&\leq&
		\|S(t,s+nl)\|_{\mathcal{L}(X)} K_s(0)  e^{-\alpha n}\\
		&=& K(s) e^{\alpha(t-nl-s)/l} \|S(t,s+nl)\|_{\mathcal{L}(X)} e^{-\alpha(t-s)/l}\\
		&\leq& 
		De^{\nu |s|}\, e^{\alpha} \, e^{\nu|t|} \, L(\nu,l) \, e^{-\alpha(t-s)/l},
	\end{eqnarray*}
	where was used the fact that $0\leq t-s-nl<l$.
	\par Now, note that, if $t\geq s\geq 0$ we have
	\begin{equation*}
	\nu|t|-\alpha(t-s)/l= -(\alpha-\nu l)(t-s)/l +\nu |s|,
	\end{equation*}
	and, for $s\leq t\leq 0$,
	\begin{equation*}
	\nu|t|-\alpha(t-s)/l= -(\alpha+\nu l)(t-s)/l +\nu |s|,
	\end{equation*}
	then choose $\hat{\alpha}=(\alpha-\nu l)/l$.
	Thus, we obtain for $t\geq s\geq 0$ and 
	$s\leq t\leq 0$ that 
	\begin{eqnarray*}
		\|S(t,s)(Id_X-Q(s))\|_{\mathcal{L}(X)}
		&\leq& 
		De^{\nu |s|}\, e^{\alpha} \, e^{\nu|t|} \, L(\nu,l) \, e^{-\alpha(t-s)/l}\\
		&\leq &
		DL(\nu,l)\, e^\alpha \, e^{2\nu |s|} e^{-\hat{\alpha}(t-s)}.
	\end{eqnarray*}
	Finally, for $t\geq 0\geq s$ we have
	\begin{eqnarray*}
		\|S(t,s)(Id_X-Q(s))\|_{\mathcal{L}(X)}
		&=& 
		\|S(t,s)(Id_X-Q(s)^2)\|_{\mathcal{L}(X)}\\
		&\leq &
		\|S(t,0)(Id_X-Q(0)\|_{\mathcal{L}(X)}\, \|S(0,s)(Id_X-Q(s))\|_{\mathcal{L}(X)}\\
		&\leq &
		D^2L(\nu,l)^2\, e^{2\alpha} \, e^{2\nu |s|} e^{-\hat{\alpha}(t-s)}.
	\end{eqnarray*}
	Therefore, for $t\geq s$,
	\begin{equation*}
	\|S(t,s)(Id_X-Q(s))\|_{\mathcal{L}(X)}\leq \
	D^2 e^{2\alpha} \max\{L(\nu,l), L(\nu,l)^2\}
	e^{2\nu|s|} e^{-\hat{\alpha}(t-s)}
	\end{equation*}
	and the first claim is proved.
	\par Now, to prove the other inequality, for $t<s$, 
	we take $n\leq 0$
	such that 
	$s+nl\leq t<s+(n+1)l$, and define for $z\in R(Q(s))$ the linear operator
	\begin{equation*}
	S(t,s)z:=S(t,s+nl)\circ [T_{0,n}(s)|_{R(Q_n(s))}]^{-1}z.
	\end{equation*}
	In other words,
	\begin{equation*}
	S(t,s)z=S(t,s+nl)\circ T_{n,0}(s)z.
	\end{equation*}
	\textbf{Claim 2:} If $t<s$, we have
	\begin{equation*}
	\|S(t,s)Q(s)\|_{\mathcal{L}(X)} \leq \hat{K}(s) 
	e^{\hat{\alpha}(t-s)}.
	\end{equation*}
	Indeed, for $x\in X$ and
	$s+nl\leq t<s+(n+1)l$, for $n\leq 0$,
	by hypothesis,
	\begin{equation*}
	\|T_{n,0}(s)Q_0(s)x\|_X\leq K_s(0) e^{\alpha n} \|x\|_X.
	\end{equation*}
	Hence, by a similar argument to that in the proof of Claim 1 we obtain that
	\begin{equation*}
	\|S(t,s)Q(s)x\|_X
	\leq 
	\|S(t,s+nl)\|_{\mathcal{L}(X)} De^{\nu |s|}  e^{\alpha n} \|x\|_X
	\leq 
	\hat{K}(s) 
	e^{\hat{\alpha}(t-s)}\|x\|_X.
	\end{equation*}
	Now, to conclude the assertion we take the supremum for $\|x\|_X=1$.
	\par \textbf{Claim 3:} For all $t_0\in \mathbb{R}$ we characterize the kernel of 
	$Q(t_0)$, $N(Q(t_0))=\{z\in X: Q(t_0)z=0 \}$, as
	\begin{equation*}
	N(Q(t_0))=\{z\in X\, : [t_0,+\infty) \ni t \mapsto S(t,t_0)z 
	\hbox{ is bounded} \}.
	\end{equation*}
	Let $z\in N(Q(t_0))$, so by definition $Q(t_0)z=0$ and for $t\geq t_0$ we can use  Claim 1 to obtain
	\begin{equation*}
	\|S(t,t_0)z\|_X =
	\|S(t,t_0)(Id_X-Q(t_0))z\|_X
	\leq 
	\hat{K}(t_0)e^{-\hat{\alpha} (t - t_0)}\|z\|_X.
	\end{equation*}
	Therefore, $[t_0,+\infty) \ni t \mapsto S(t,t_0)z$ is bounded.
	\par On the other hand, if $z\notin N( Q( (t_0) ) )$ and 
	$n>0$,
	\begin{eqnarray*}
	\|Q(t_0)z\|_X &\leq &
	\|T_{0,n}(t_0) Q_n(t_0)\|_{\mathcal{L}(X)}
	\|T_{n,0}(t_0)z\|_X\\
	&\leq& De^{\nu |t_0|} e^{\nu |n|}e^{-\alpha n}\|S(t_0+nl,t_0)z\|_X.
	\end{eqnarray*}
	Thus, we obtain
	\begin{equation*}
	\|Q(t_0)z\|_X D^{-1} e^{-\nu |t_0|} e^{n(\alpha -\nu)}
	\leq \|S(t_0+nl,t_0)z\|_X.
	\end{equation*}
	Consequently, as $\nu<\alpha$ we have that 
	$[t_0,+\infty) \ni t \mapsto S(t,t_0)z$ is not bounded.
	\par Note that the last assertion implies that
	\begin{equation*}
	S(t,t_0)N(Q(t_0)) \subset N(Q(t)).
	\end{equation*}
	
	\par \textbf{Claim 4:} The linear operator
	\begin{equation*}
	S(t,t_0): R(Q(t_0))\rightarrow X
	\end{equation*}
	is injective for all $t\geq t_0$.
	\par Indeed, let $z\in R(Q(t_0))$ with $S(t,t_0)z=0$.
	Choose $n\in \mathbb{N}$ so that $t\leq nl+t_0$, then
	\begin{equation*}
	0=S(t_0+nl,t)0=S(t_0+nl,t)S(t,t_0)z=T_{n,0}(t_0)z,
	\end{equation*}
	this implies that $z\in N(T_{n,0}(t_0)|_{R(Q_0(t_0))})=\{0\}$.
	\par \textbf{Claim 5:} For all $t_0\in \mathbb{R}$ the range of 
	$Q(t_0)$ is
	\begin{equation*}
	R(Q(t_0))=\{z\in X\,: \hbox{ there exists a backwards bounded solution }\xi\hbox{ with } \xi(t_0)=z\}.
	\end{equation*}
	Let $z\in R(Q(t_0))$ and $t<t_0$, then take $n\in \mathbb{Z}$ such that 
	$t\in [t_0+nl,t_0+(n+1)l]$ and define
	\begin{equation*}
	\xi(t):=S(t,t_0+nl)T_{n,0}(t_0)z=S(t,t_0)z.
	\end{equation*}
	Now, choose $x\in X$ so that $z=Q(t_0)x$, thus by Claim 2
	\begin{equation*}
	\|\xi(t)\|_X \leq \hat{K}(t_0) 
	e^{\hat{\alpha}(t-t_0)} \|x\|_X.
	\end{equation*}
	Thus, $\xi$ is a backward bounded solution such that $\xi(t_0)=z$.
	Suppose that $z\notin R(Q(t_0))$ and that there exists 
	$\xi: \mathbb{R}\rightarrow X$ a global solution such that $\xi(t_0)=z$.
	For $n\leq 0$ we can write 
	$z=S(t_0,t_0+nl)\xi(t_0+nl)$, 
	thus
	\begin{eqnarray*}
		\|(Id_X-Q(t_0)) z\|_X &\leq& 
		\|S(t_0,t_0+nl)(Id_X-Q(t_0+nl))\|_{\mathcal{L}(X)} \,
		\|\xi(t_0+nl)\|_X \\
		&\leq& 
		De^{\nu |t_0|} e^{\nu |n|} e^{\alpha n} \|\xi(t_0+nl)\|_X.
	\end{eqnarray*}
	Therefore, 
	\begin{equation*}
	\|(Id_X-Q(t_0)) z\|_X 
	D^{-1}e^{-\nu |t_0|}e^{n(\nu-\alpha)} \leq\|\xi(t_0+nl)\|_X.
	\end{equation*}
	Since $\nu<\alpha$, it follows that $\xi$ is not backwards bounded, and the proof of Claim 5 is complete.
		
	\par \textbf{Claim 6:} $S(t,t_0)R(Q(t_0))= R(Q(t))$.
	\par Indeed, if $z\in R(Q(t_0))$, then there exists 
	a backwards bounded solution $\xi$ through $z$ in $t=t_0$.
	Thus, $\xi$ is also a solution through 
	$S(t,t_0)z$ in time $t$ and we see that 
	$S(t,t_0)z\in R(Q(t))$.
	On the other hand, if $z\in R(Q(t))$, there is a backwards bounded solution $\xi$ with
	$\xi(t)=z$. Therefore, if $n\in \mathbb{Z}$ such that
	$nl+t\leq t_0\leq t$, define 
	\begin{equation*}
	x=S(t_0,nl+t)S(nl+t,t)z\in R(Q(t_0)).
	\end{equation*}
	Therefore, $S(t,t_0)x=z$ and we conclude that $S(t,t_0)|_{R(Q(t_0))}$ is an isomorphism.
	
	\par Finally, we prove that the family of projections commutates with the evolution process.
	\par \textbf{Claim 7:} $Q(t)S(t,s)=S(t,s)Q(s)$.
	For $z\in X$, we have that
	\begin{equation*}
	S(t,t_0)z=S(t,t_0)(Id_X-Q(t_0))z + S(t,t_0)Q(t_0)z.
	\end{equation*}
	Now, as $(Id_X-Q(t_0))z\in N(Q(t_0))$ and
	$S(t,t_0)Q(t_0)z\in R(Q(t))$, applying $Q(t)$ we obtain
	\begin{equation*}
	Q(t)S(t,t_0)z=S(t,t_0)Q(t_0)z.
	\end{equation*}
\end{proof}

\par We are ready to present the main result of this section. 

\begin{theorem}[Robustness for continuous evolution processes]\label{th-roughness-continuous-TED}
	Let $\mathcal{S}=\{S(t,s):t\geq s\}\subset \mathcal{L}(X)$ be an evolution process that
	admits a nonuniform exponential dichotomy 
	with bound $K(s)=De^{\nu |s|}$ and exponent 
	$\alpha>\nu$. Assume that	
	\begin{equation}\label{th-roughness-continuous-TED-hypothesis1}
	L_\mathcal{S}(\nu):=\sup_{0\leq t-s\leq 1} \big\{e^{-\nu |t|} \|S(t,s)\|_{\mathcal{L}(X)}\big\}<+\infty.
	\end{equation}
	Then there exists $\epsilon>0$ such that if 
	$\mathcal{T}=\{T(t,s)\, :\, t\geq s\}$
	is an evolution process such that
	\begin{equation}\label{th-roughness-continuous-TED-hypothesis2}
	\sup_{0\leq t-s\leq 1}\big\{ K(t) \|S(t,s)-T(t,s)\|_{\mathcal{L}(X)} \big\}<\epsilon,
	\end{equation}
	then $\mathcal{T}$ admits a nonuniform exponential dichotomy with 
	exponent 
	$\hat{\alpha}:=\tilde{\alpha}-\nu$ and
	bound
	\begin{equation}\label{th-hat-K}
	\hat{K}(s)=
	\tilde{D}^2 e^{2\tilde{\alpha}} \max\{L_\mathcal{T}(\nu), L_\mathcal{T}(\nu)^2\}
	\,	e^{2\nu|s|},
	\end{equation}
	where
	$\tilde{D}:=D(1+\epsilon/(1-\rho)(1-e^{-\alpha}))\max\{D_1,D_2\}$, and
	$\rho,\tilde{\alpha}, D_1$ and $D_2$ are the same as in 
	Theorem \ref{th-roughness-discrete-TED}.
\end{theorem}

\begin{proof}
	Let $n\in \mathbb{Z}$ and $t_0\in \mathbb{R}$, then, by 
	Theorem \ref{th-continuousTED-implies-discreteTED},
	the discrete evolution process
	$\{S_{n}(t_0):=S(t_0+n+1,t_0+n) \, : \, n\in \mathbb{Z}\}$ 
	admits a nonuniform exponential dichotomy with bound 
	$K_t(n)\leq De^{\nu (|t+n|)}$
	and exponent $\alpha>0$.
	Let $\epsilon>0$ be such that $\epsilon<(1-e^{-\alpha})/(1+e^{-\alpha})$ and 
	$\mathcal{T}=\{T(t,s): t\geq s\}$ an evolution process that satisfies 
	\eqref{th-roughness-continuous-TED-hypothesis2}.
	Let $\{T_n(t_0):n\in \mathbb{Z}\}$ be the discretization of $\mathcal{T}$ at $t_0$ and
	define, for each $n\in\mathbb{Z}$ and $t_0\in \mathbb{R}$, the linear bounded operator
	 $$B_n(t_0):=T_n(t_0)- S_n(t_0).$$
	Hence, from \eqref{th-roughness-continuous-TED-hypothesis2}, we have that
	\begin{equation*}
	\|B_n(t_0)\|_{\mathcal{L}(X)}<
	\epsilon K_{t_0}(n+1)^{-1}.
	\end{equation*}
	Therefore, by Theorem \ref{th-roughness-discrete-TED}, 
	the discrete evolution process
	$T_n(t_0)=S_n(t_0)+B_n(t_0)$ 
	admits a nonuniform exponential dichotomy 
	with exponent
	\begin{equation*}
	\tilde{\alpha}:=-\ln(\cosh \alpha - [\cosh ^2 \alpha-1-2\epsilon\sinh \alpha]^{1/2}),
	\end{equation*}
	and bound
	\begin{equation*}
	\tilde{K}_{t_0}(n):=K_{t_0}(n)\left[1+\frac{\epsilon}{(1-\rho)(1-e^{-\alpha})}\right]
	\max{[D_1,D_2]},
	\end{equation*}
	where $D_1,D_2,\rho$ are constants that can be found in Theorem 
	\ref{th-roughness-discrete-TED}.
	\par Since each discretization at time $t$ have the same exponent $\alpha>0$ we see that $\epsilon$ can be choose independent of $t$. Thus
	for each $t\in \mathbb{R}$, the discrete evolution process 
	$\{T_n(t): n\in \mathbb{Z}\}$ admits nonuniform exponential dichotomy with bound $\tilde{K}_{t}(n)$ and 
	exponent $\tilde{\alpha}$ defined above. Then condition (2) of Theorem 
	\ref{th-discrete-dichotomy-implies-continuous-dichotomy} hold true for $\mathcal{T}$.
	\par Moreover, from \eqref{th-roughness-continuous-TED-hypothesis2}, $\mathcal{T}$ satisfies
	\begin{eqnarray*}
	\|T(t,s)\|_{\mathcal{L}(X)}&\leq& \epsilon K(t)^{-1} +\|S(t,s)\|_{\mathcal{L}(X)}\\
	&\leq &\epsilon  +\|S(t,s)\|_{\mathcal{L}(X)}, \hbox{ for } 0\leq t-s\leq 1
	\end{eqnarray*}
	then
	$\sup_{ 0\leq t-s\leq 1}\{e^{-\nu |t|}\|T(t,s)\|_{\mathcal{L}(X)}\}$
	is finite.
	Finally, note that it is possible to choose $\epsilon>0$ small such that 
	$\tilde{\alpha}>\nu$.
	Therefore, Theorem 
	\ref{th-discrete-dichotomy-implies-continuous-dichotomy} implies that
	$\mathcal{T}$ admits nonuniform exponential dichotomy with 
	bound $\hat{K}$ defined in \eqref{th-hat-K} and exponent $\hat{\alpha}=\tilde{\alpha}-\nu>0$. 
\end{proof}

\begin{remark}
	Assumption \eqref{th-roughness-continuous-TED-hypothesis1} on the growth of $\mathcal{S}$ is expected for evolution processes that admit nonuniform exponential dichotomies, see Barreira and Valls \cite{Barreira-Valls-Sta} or Example \ref{example-nonuniformED} in Section \ref{subsection-a-general-example}.
\end{remark}

\begin{remark}
	Theorem \ref{th-roughness-continuous-TED} allows us to see the robustness as an \textit{open property}. In fact, let $\mathfrak{S}_\nu$ be the space every evolution process that
	satisfy 
	\eqref{th-roughness-continuous-TED-hypothesis1} 
	and define a distance in $\mathfrak{S}_\nu$ as
	\begin{equation*}
	d_\nu(\mathcal{S},\mathcal{T}):=\sup_{0\leq t-s\leq 1}\big\{ e^{\nu|t|} \|S(t,s)-T(t,s)\|_{\mathcal{L}(X)} \big\}.
	\end{equation*}
	Then, from Theorem \ref{th-roughness-continuous-TED} we see that if $\mathcal{S}\in \mathfrak{S}_\nu$ admits a nonuniform exponential dichotomy with bound $K(t)=De^{\nu|t|}$ and exponent $\alpha>\nu$, then there exists $\epsilon>0$ such that
	every evolution process $\mathcal{T}$ in a $\epsilon$-neighborhood of $\mathcal{S}$
	admits a nonuniform exponential dichotomy with bound and exponent given in Theorem \ref{th-roughness-continuous-TED}. 
\end{remark}

\par Now, we present another formulation of Theorem \ref{th-roughness-continuous-TED}
that allows us to apply the result for differential equations in Banach spaces.
%

\begin{theorem}\label{th-perturation-of-nununiform-exp-dcihotomy-for-differential-eq}
	Let $\mathcal{S}=\{S(t,s):t\geq s\}\subset \mathcal{L}(X)$ be an evolution process that
	admits a nonuniform exponential dichotomy 
	with bound $K(s)=De^{\nu |s|}$ and exponent 
	$\alpha>\nu$. Assume that	
	\begin{equation*}\label{th-roughness-continuous-TED-H1}
	L_\mathcal{S}(\nu):=\sup_{0\leq t-s\leq 1} \big\{e^{-\nu |t|} \|S(t,s)\|_{\mathcal{L}(X)}\big\}<+\infty.
	\end{equation*}
	Let $\{B(t)\, :\, t\in \mathbb{R}\}\subset \mathcal{L}(X)$
	so that $\mathbb{R} \ni t\mapsto B(t)x$ is continuous for all 
	$x\in X$ and
	\begin{equation*}
	\|B(t)\|_{\mathcal{L}(X)}<\delta e^{-3\nu |t|}.
	\end{equation*}
	Then any evolution process that satisfies the integral equation
		\begin{equation}\label{eq-th-perturbed-equation-VCF}
		T(t,s)=S(t,s)+\int_{s}^{t}S(t,\tau)B(\tau) T(\tau,s)d\tau  \in \mathcal{L}(X),  \ \ t\geq s,
		\end{equation}
	admits a nonuniform exponential dichotomy for suitably small $\delta>0$, with bound and exponent given in Theorem \ref{th-roughness-continuous-TED}.
\end{theorem}
\begin{proof}
	Let $\mathcal{T}=\{T(t,s): t\geq s\}$ be a evolution process satisfying \eqref{eq-th-perturbed-equation-VCF}. Then 
	\begin{equation*}
	\|T(t,s)\|_{\mathcal{L}(X)}\leq \|S(t,s)\|_{\mathcal{L}(X)}+\int_{s}^{t} \|S(t,\tau)\|_{\mathcal{L}(X)} \|B(\tau)\|_{\mathcal{L}(X)} \, \|T(\tau,s)\|_{\mathcal{L}(X)}d\tau.
	\end{equation*}
	Thus, fix $s$ and define the function
	$\phi(t)=e^{-\nu|t|}\|T(t,s)\|_{\mathcal{L}(X)}$, for $t\leq s+1$,
	\begin{equation*}
	\phi(t)\leq L_\mathcal{S}(\nu)+ L_\mathcal{S}(\nu)\int_{s}^{t} \|B(\tau)\|_{\mathcal{L}(X)}e^{\nu|\tau|}\phi(\tau) d\tau
	\end{equation*}
	By Grownwall's inequality, we obtain
	that
	\begin{equation*}
	\phi(t)\leq L_\mathcal{S}(\nu) 
	e^{L_\mathcal{S}(\nu) \int_s^t\|B(\tau)\|_{\mathcal{L}(X)}e^{\nu|\tau|}d\tau}, \hbox{ for } t\leq s+1.
	\end{equation*}
	Therefore, 
	\begin{equation*}
	L_\mathcal{T}(\nu):=\sup_{ 0\leq t-s\leq 1}\big\{e^{-\nu|t|}\|T(t,s)\|_{\mathcal{L}(X)}\big\} <+\infty.
	\end{equation*}
	Now, 
	for
	$0\leq t-s\leq 1$, 
	\begin{eqnarray*}
		\|S(t,s)-T(t,s)\|_{\mathcal{L}(X)}
		&\leq& \int_{s}^{t} e^{\nu(|t|+|\tau|)}  L_\mathcal{S}(\nu)\|B(\tau)\|_{\mathcal{L}(X)} 
		L_\mathcal{T}(\nu) d\tau\\
		&= &L_\mathcal{T}(\nu) L_\mathcal{S}(\nu) \, e^{\nu|t|}\int_{s}^{t}e^{\nu|\tau|}\|B(\tau)\|_{\mathcal{L}(X)}d\tau .
	\end{eqnarray*}
	Then
	\begin{equation*}
	K(t)\|S(t,s)-T(t,s)\|_{\mathcal{L}(X)}
	\leq L_\mathcal{T}(\nu) L_\mathcal{S}(\nu) D\, \delta,
	\end{equation*}
	and choose $\delta >0$ suitably small in order to use Theorem
	\ref{th-roughness-continuous-TED}
	and conclude the proof.
\end{proof}

%
%
%
	\par Theorem \ref{th-perturation-of-nununiform-exp-dcihotomy-for-differential-eq}
	is very useful when dealing with differential equations.
	In fact, let $\{A(t) : \, t\in \mathbb{R}\}$ be a family of linear operators, bounded or unbounded, and consider
	\begin{equation}\label{eq-standart-linear-equation}
	\dot{x}=A(t)x, \ \ x(s)=x_s\in X.
	\end{equation} 
	Suppose that for each $s\in \mathbb{R}$ and $x_s\in X$ there exists a unique solution $x(\cdot,s,x_s):[s,+\infty)\to X$. Thus there exists an  evolution process $\mathcal{S}=\{S(t,s):t\geq s\}$ defined by $S(t,s)x_s:=x(t,s,x_s)$ for each $t\geq s$.
	\par To study robustness of nonuniform exponential dichotomy of problem
	\eqref{eq-standart-linear-equation}, 
	we suppose that
	$\mathcal{S}$ admits a nonuniform exponential dichotomy and we want to know  
	for which class of 
	$\{B(t): \, t\in \mathbb{R}\}\subset \mathcal{L}(X)$
	the perturbed problem
	\begin{equation}\label{eq-standart-linear-equation-perturbed}
	\dot{x}=A(t)x+B(t)x, \ \ x(s)=x_s\in X,
	\end{equation}	
	admits a nonuniform exponential dichotomy with bound $K(t)=De^{\nu|t|}$ and exponent $\alpha>0$.
	In this way, Theorem \ref{th-perturation-of-nununiform-exp-dcihotomy-for-differential-eq} ensures that
	the nonuniform hyperbolicity is preserved for exponentially small perturbations.
	In other words, if the norm of the perturbation of $B$ does not grow more than $e^{-3\nu |t|}$ for $\nu<\alpha$, then the perturbed problem
	\eqref{eq-standart-linear-equation-perturbed} admits a nonuniform exponential dichotomy.
	
	\begin{remark}
		\par In Barreira and Valls \cite{Barreira-Valls-Robustness-noninvertible} 
		is also provide a version of Theorem
			\ref{th-perturation-of-nununiform-exp-dcihotomy-for-differential-eq} under different assumptions.
		They considered a general growth rate $\rho(t)$ for the nonuniform exponential dichotomy and proved that if $\alpha >2\nu$ and 
		$B:\mathbb{R}\to \mathcal{L}(X)$ is continuous satisfying
		$\|B(t)\|_{\mathcal{L}(X)}\leq \delta e^{-3\nu |\rho(t)| }\rho^\prime(t)$, 
		for all $t\in \mathbb{R}$, then the perturbed problem 
		\eqref{eq-standart-linear-equation-perturbed} admits $\rho$-nonuniform exponential dichotomy. 
		We note that our method does not work for general growth rates $\rho(t)$. On the other hand, for $\rho(t)=t$, since our condition on the exponents is only $\alpha>\nu$ we obtain a improvement of their robustness result (at this particular case).
	\par Considering invertible evolution processes, Barreira and Valls proved in \cite{Barreira-Valls-R} a result similar to
	Theorem \ref{th-perturation-of-nununiform-exp-dcihotomy-for-differential-eq}.
	It is assumed that $A(t)$ is bounded, so the evolution process 
	$\mathcal{S}$
	is invertible (which means that $S(t,s)$ is invertible for every $t\geq s$) and
	that the perturbation $B$ satisfies
	$\|B(t)\|_{\mathcal{L}(X)}\leq \delta e^{-2\nu|t|}$ for all $t$.
	In their proof, thanks to invertibility, they can write explicit expressions of
	the projections for the perturbed evolution process. However, in many situations, 
	it is not expected to have that $A(t)\in \mathcal{L}(X)$, see  \cite{Carvalho-Langa-Robison-book,Henry-1}. 
	\end{remark}

\section{An application in infinite-dimensional differential equations}\label{subsection-a-general-example}
\par In this section, we show an application of the robustness result in order to obtain examples of evolution processes that admits nonuniform exponential dichotomies. 
Inspired in an example of \cite{Barreira-Valls-Sta}, 
we provide an evolution process defined on a Banach space that admits a nonuniform exponential dichotomy. Then, apply Theorem \ref{th-perturation-of-nununiform-exp-dcihotomy-for-differential-eq} 
to study for which class of perturbations the nonuniform hyperbolicity will be preserved.
\par Let $X$ and $Y$ be two Banach spaces. 
Suppose that $A$ is a generator of a $C_0$-semigroup 
$\{e^{At}: t\geq 0\}$ in $X$ and $B\in \mathcal{L}(Y)$ with
$\sigma(A)\subset (-\infty,-\omega)$
and $\sigma(B)\subset (\omega,+\infty)$, for some $\omega >0$, and there exists 
$M\geq 1$ such that
\begin{eqnarray*}
	\|e^{A (t-s)}\|_{\mathcal{L}(X)} &\leq& M e^{-\omega(t-s)}, \, t\geq s; \\ 
	\|e^{B (t-s)}\|_{\mathcal{L}(Y)} &\leq& M e^{\omega(t-s)},\,  t<s.
\end{eqnarray*}
\begin{remark}
	Let $\mathcal{C}$ be a generator of an \textit{hyperbolic $C_0$-semigroup} $\{e^{\mathcal{C}t}: t\geq 0\}$, i.e., the associated evolution processes  $\{e^{\mathcal{C}(t-s)}: t\geq s\}$ admits an uniform exponential dichotomy with a single projection $Q(t)=Q\in \mathcal{L}(X)$ for every $t\in \mathbb{R}$. Then, there is a decomposition
	$X=X^u\oplus X^s$ such that $A:=C|_{X^s}$ and $B=\mathcal{C}|_{X^u}$ satisfy the conditions above over $X^s$ and $X^u$, respectively, see \cite{Carvalho-Langa-Robison-book,Chow-Leiva-1,Henry-1}.
\end{remark}
\par Let $\omega>a>0$ and define the linear operator in $Z=X\times Y$
\begin{equation*}
\mathcal{A}(t):=\left[ {\begin{array}{cc}
	A -at \sin(t)Id_{X} & 0 \\
	0 & B + at\sin(t)Id_{Y} \\
	\end{array} } \right].
\end{equation*}
Consider the differential equation 
\begin{equation}\label{example-nonuniformED}
\dot{z}=\mathcal{A}(t)z, \ \ z(s)=z_s\in Z.
\end{equation}
\par Then, the evolution process associated with problem \eqref{example-nonuniformED} is defined by
\begin{equation*}
T(t,s)=(U(t,s), V(t,s))
\end{equation*}
where 
\begin{eqnarray*}
U(t,s)&=&e^{A (t-s)} \exp\bigg\{-\int_{s}^{t}a\tau \sin(\tau ) d\tau \bigg\} \hbox{ and}\\
V(t,s)&=&e^{B (t-s)} \exp\bigg\{\int_{s}^{t}a\tau \sin(\tau ) d\tau \bigg\}
\end{eqnarray*}
are evolution processes in $X$ and $Y,$ respectively.
\par The proof of the next result is inspired by Proposition 2.3 of \cite{Barreira-Valls-Sta}.
\begin{proposition}\label{th-general-example-of-non-uniform-exp-dichotomy}
	Let $\mathcal{T}=\{T(t,s):t\geq s\}$ be the evolution process defined above. Then
	$\mathcal{T}$ admits a nonuniform exponential dichotomy with bound $K(t)=Me^{2a(1+|t|)}$ and exponent
	$\alpha=\omega-a>0$.
\end{proposition}
\begin{proof}
	\par Define the linear operators 
	$P(t)=P_X$ and $Q(t)=P_Y$ for all $t\in \mathbb{R}$
	where $P_X$ and $P_Y$ are the canonical projections onto $X$ and $Y$,
	respectively.
	Then $T(t,s)P(s)=U(t,s)$ and $T(t,s)Q(s)=V(t,s)$ for all $t\geq s$. 
	\par In this way we have that $P_X$ commutates with $T(t,s)$,
	for all $t\geq s$ and 
	since $B\in\mathcal{L}$(Y) generates a group in $Y$ we have that 
	$V(t,s)$ is an isomorphism over $Y.$ 
	Note that
	\begin{eqnarray*}
	\|U(t,s)\|_{\mathcal{L}(X)}&=&\exp\bigg\{-\int_{s}^{t}a\tau \sin(\tau ) d\tau \bigg\} \|e^{A (t-s)}\|_{\mathcal{L}(X)}\\
	&\leq& Me^{-\omega(t-s)+at\cos(t)  -as\cos(s) -a\sin(t) +a\sin(s)}
	\end{eqnarray*}
	Now, proceed as in Proposition 2.3 of \cite{Barreira-Valls-Sta} to obtain
	\begin{equation}\label{eq-estimate-for-U}
	U(t,s)\leq  Me^{(-\omega+a)(t-s)+2a|s|+2a}, \hbox{ for } t\geq s.
	\end{equation}
	Similarly, we obtain that 
	\begin{equation}\label{eq-estimate-for-V}
	\|V(t,s)\|_{\mathcal{L}(Y)}
	\leq Me^{2a+2a|s|} e^{(\omega+a)(t-s)} \hbox{ for } t<s.
	\end{equation}
	Therefore, $\mathcal{T}$ admits a nonuniform exponential dichotomy with
	bound $K(t)=Me^{2a(1+|t|)}$ and exponent
	$\alpha=\omega-a>0$.
\end{proof}


\par Now, apply Theorem \ref{th-perturation-of-nununiform-exp-dcihotomy-for-differential-eq} to  Example \ref{example-nonuniformED}.

\begin{theorem}\label{th-example-application-of-robustness}
	Consider for each $\epsilon>0$ the operator $B_\epsilon(t)\in\mathcal{L}(Z)$ such that
	$\|B_\epsilon(t)\|\leq \epsilon e^{-6a |t|}$,
	and define the operator 
	\begin{equation*}
	\mathcal{A_\epsilon}(t):=\mathcal{A}(t) +B_\epsilon(t).
	\end{equation*}
	If $\omega>3a$, there exists $\epsilon>0$ such that the evolution process associated with the problem
	\begin{equation}\label{eq-application-section-an-application}
	\dot{x}=\mathcal{A}_\epsilon(t)x, \ \ x(s)=x_s\in Z.
	\end{equation}
	admits a nonuniform exponential dichotomy.
\end{theorem}

\begin{proof}
	Let us prove first that the evolution problem associated with \eqref{example-nonuniformED} satisfies
	\begin{equation}\label{eq-norm-evolution-process-last-theorem}
	\sup_{ 0\leq t-\tau\leq 1} \big\{e^{-\nu |t|}\|T(t,\tau)\|_{\mathcal{L}(Z)} \big\} <+\infty.
	\end{equation}
	In fact, we have for $t\geq s$ that
	\begin{equation*}
	\|T(t,s)\|_{\mathcal{L}(Z)}\leq \|U(t,s)\|_{\mathcal{L}(X)} + \|V(t,s)\|_{\mathcal{L}(Y)},
	\end{equation*}
	where $U$ and $V$ are the evolution processes defined in the proof of Proposition
	\ref{th-general-example-of-non-uniform-exp-dichotomy}.
	Then it is enough to prove that each evolution process satisfies 
	\eqref{eq-norm-evolution-process-last-theorem}
	in the corresponding space.
	From 
	\eqref{eq-estimate-for-U}
	we have that
	\begin{equation*}
	e^{-2a|t|}\|U(t,s)\|_{\mathcal{L}(X)}
	\leq Me^{2a+2a(|s|-|t|)} e^{-(\omega-a)(t-s)}.
	\end{equation*}
	Thus, we have to analyze the term 
	$e^{-2a(|s|-|t|)}$ for $0\geq t-s\geq 1$. 
	If $t\geq s\geq 0$ or $0\geq t\geq s$ we have
	$|s|-|t|=|s-t|$, then 
	$e^{2a(|s|-|t|)}\leq e^{2a|s-t|}$ is a continuous function bounded on the set 
	$\{(t,s): 0\leq t-s\leq 1\}$. 
	Also, if 
	$t\geq 0\geq s$ we have
	$|s|-|t|=-s-t=-(s-t)-2t\leq 1$ so
	$e^{2a(|s|-|t|)}\leq e^{2a}$ and therefore
	\begin{equation*}
	\sup_{ 0\leq t-s\leq 1}\{e^{-2a|t|}\|U(t,s)\|_{\mathcal{L}(X)}\}
	<+\infty, \hbox{ for all } t\geq s.
	\end{equation*}
	Note that
	$\|e^{B(t-s)}\|_{\mathcal{L}(Y)}\leq \tilde{M}e^{\beta(t-s)}$
	for some $\tilde{M}\geq 1$ and $\beta>0$, for every $t\geq s$.
	Then
	\begin{equation*}
	\|V(t,s)\|_{\mathcal{L}(Y)}= \exp\bigg\{\int_{s}^{t}a\tau \sin(\tau ) d\tau \bigg\} \|e^{B (t-s)}\|_{\mathcal{L}(Y)}
	\leq \tilde{M}^2e^{4a+2a|t|} e^{(\beta+a)(t-s)},
	\end{equation*}
	which implies that
	\begin{equation*}
	\sup_{ 0\leq t-s\leq 1}\{ e^{-2a|t|}\|V(t,s)\|_{\mathcal{L}(Y)}\}<+\infty.
	\end{equation*}
	\par Now, from Proposition \ref{th-general-example-of-non-uniform-exp-dichotomy},
	$\mathcal{T}$ admits a nonuniform 
	exponential dichotomy
	where the bound is $K(s)=Me^{2a+2a|s|}$ and exponent $\alpha=\omega-a>0$. 
	 Since $\nu:=2a$ is such that $\alpha>\nu$, we apply Theorem
	\ref{th-perturation-of-nununiform-exp-dcihotomy-for-differential-eq}
	to conclude that the evolution process generated by \eqref{eq-application-section-an-application} admits a nonuniform exponential dichotomy.
\end{proof}

\begin{remark} 
	Note that, in Theorem \ref{th-perturation-of-nununiform-exp-dcihotomy-for-differential-eq}
	 the assumption $\alpha>\nu$ of Theorem \ref{th-example-application-of-robustness} is expressed by $\omega>3a$. On the other hand, to apply 
	Theorem 1 of
	\cite{Barreira-Valls-Robustness-noninvertible}
	the hypothesis must be $\omega>5a$, because their condition is $\alpha>2\nu$.
\end{remark}

\section{Persistence of nonuniform hyperbolic solutions}\label{section-persistence}

\par In this section, we study nonlinear evolution processes associated with a semilinear differential equation. Inspired by \cite{Carvalho-Langa-Robison-book},
we study persistence of \textit{nonuniform hyperbolic solutions} under perturbation for evolutions processes in Banach spaces. More precisely, we
 use \textit{Greens function} to characterize 
bounded global solutions for semilinear differential equations and conclude that 
nonuniform hyperbolic solutions are \textit{isolated} in the set of bounded continuous functions, see Theorem \ref{lemma-non-uniform-hperbolic-solutions-second}. 
Finally, in Theorem \ref{th-persistence}, 
we provide conditions to prove that 
nonuniform hyperbolic solutions persist under perturbations.
\par Consider a semi-linear differential equation
\begin{equation}\label{nonuniform-hyperbolic-solutions-definition}
\dot{y}=A(t)y +f(t,y), \ \ y(s)=y_s.
\end{equation}
Assume that $f$ is continuous in the first variable and 
locally Lipschitz in the second and that
$A(t)$ is associated with a linear bounded evolution process 
$\mathcal{T}=\{T(t,s): t\geq s\}$, i.e., for each $s\in\mathbb{R}$ and $x\in X$ the mapping
$[s,+\infty)\ni t\to T(t,s)x$ is the solution of 
$$\dot{x}=A(t)x, \ x(s)=x.$$
Then we have a \textit{local mild solution} for problem 
\eqref{nonuniform-hyperbolic-solutions-definition},
that is, for each $(s,y_s)\in  \mathbb{R}\times X$ there exist
$\sigma=\sigma(s,y_s)>0$ and a solution $y$ of the integral equation
\begin{equation}
y(t,s;y_s)=T(t,s)y_s+\int_{s}^{t}T(t,\tau)f(\tau,y(\tau,s)) d\tau, 
\end{equation}
for all $t\in [s, s+\sigma)$.
\par If for each $(s,y_s)\in  \mathbb{R}\times X$,
$\sigma(s,y_s)=+\infty$, we can consider the evolution process
$S_f(t,s)y_s=y(t,s;y_s)$.
We refer to $\mathcal{S}_f=\{S_f(t,s):t\geq s\}$ as the evolution process obtained 
by a non-linear perturbation $f$ of
$\mathcal{T}$.
\par Suppose additionally that 
$f:\mathbb{R}\times X \rightarrow X$
is differentiable  with continuous derivatives. Let $\xi$ be a global solution of 
$\mathcal{S}_f$ (see Definition \ref{def-global-solution}), and 
$\mathcal{L}_f=\{L_f(t,s)\,: t\geq s \}$ is the linearized evolution process
of  $\mathcal{S}_f$ on $\xi$. Thus $\mathcal{L}_f$ satisfies
\begin{equation*}
L_f(t,s)=T(t,s)+\int_{s}^{t} T(t,\tau)D_xf(\tau,\xi(\tau)) L_f(\tau,s)d\tau.
\end{equation*}
\begin{definition}
	If $\mathcal{L}_f$ admits a nonuniform exponential dichotomy we say that
	$\xi$ is a \textbf{nonuniform hyperbolic solution} for $\mathcal{S}_f$.
\end{definition}
\par In Barreira and Valls \cite{Barreira-Valls-Sta} this notion is called
\textit{nonuniformly hyperbolic trajectories}.
%

\begin{remark}
	 We point out that the existence of a nonuniform hyperbolic solution
	 can be obtained by an application of Theorem \ref{th-perturation-of-nununiform-exp-dcihotomy-for-differential-eq}. 
	 For instance, if $\mathcal{T}$ admits nonuniform exponential dichotomy with bound $K(s)=De^{\nu|s|}$ 
	 such that
	 \begin{equation*}
	 \sup_{ 0\leq t-s\leq 1}\big\{e^{-\nu|t|} \|T(t,s)\|_{\mathcal{L}(X)} \big\}<+\infty,
	 \end{equation*}
	 and $f_\epsilon$ is such that $f_\epsilon(t,\cdot)$ is differentiable with continuous derivatives, and 
	 $$\sup_{t\in \mathbb{R}}\|e^{3\nu|t|}f_\epsilon(t,\cdot)\|_{C^1(X)}\to 0 \hbox{ as }\epsilon\to 0.$$ Then, for each $\epsilon>0$ suitable small, it is possible to obtain nonuniform hyperbolic solutions $\xi_\epsilon$ for $\mathcal{S}_{f_\epsilon}$.
\end{remark}
\begin{remark}\label{lemma-non-uniform-hperbolic-solutions-first-lemma}
	Let $\varphi$ be a global solution for
	$\mathcal{S}_f$.
	Then
	\begin{equation}\label{eq-lemma-non-uniform-hperbolic-solutions-first-lemma}
	\varphi(t)=L_f(t,s)\varphi(s)+\int_{s}^{t}L_f(t,\tau)
	[f(\tau,\varphi(\tau)) - D_xf(\tau,\xi(\tau)) \varphi(\tau)] d\tau,  \ \ t\geq s.
	\end{equation} 
	In particular, the global bounded solution $\xi$ satisfies the integral equation
	\eqref{eq-lemma-non-uniform-hperbolic-solutions-first-lemma}.
\end{remark}

\par The next result allows us to characterize bounded nonuniform hyperbolic solutions.
\begin{theorem}
	\label{lemma-non-uniform-hperbolic-solutions-second}
	Assume that there is a global nonuniform hyperbolic 
	solution $\xi$ for $\mathcal{S}_f$ and that $\mathcal{L}_f$ admits a nonuniform exponential dichotomy with bound 
	is $K(s)=De^{\nu|s|}$, for all $s\in \mathbb{R}$, and exponent $\alpha>\nu$.
	If $\varphi$ is a bounded global solution for $\mathcal{S}_f$, then $\varphi$
	satisfies
	\begin{equation*}
	\varphi(t)=\int_{-\infty}^{+\infty} G_f(t,\tau)  [f(\tau,\varphi(\tau)) - D_xf(\tau,\xi(\tau)) \varphi(\tau)] d\tau,
	\end{equation*}
	where $G_f$ is the Green function associated with the evolution process
	$\mathcal{L}_f$, 
	\begin{equation*}
	G_{f}(t,s)= \left\{ 
	\begin{array}{l l} 
	L_f(t,s)(Id_X-Q(s)), 
	&  \quad \hbox{if } t\geq s,
	\\ -L_f(t,s)Q(s) \, 
	& \quad \hbox{if } t<s.
	\end{array} 
	\right.
	\end{equation*} 
	Moreover, if $\xi$ is a bounded nonuniform hyperbolic solution of $\mathcal{S}_f$
	and 
	\begin{equation}\label{equation-condition-to-obtain-persisntence-on-f}
	\rho(\epsilon)=
	\sup_{\|x\|	\leq \epsilon}\sup_{t\in \mathbb{R}} 
	\frac{e^{\nu |t|} \, \|f(t,\xi(t)+x) -f(t,\xi(t) ) -D_xf(t,\xi(t) )x\| }{\|x\|}
	\to 0, \hbox{ as } \epsilon \to 0,
	\end{equation}
	then $\xi$ is \textbf{isolated} in the set of bounded and continuous functions
	$C_b(\mathbb{R}, X)$, i.e., there is a neighborhood of $\xi$ such that
	$\xi$ is the only bounded global solution of $\mathcal{S}_f$ with the
	trace inside of this
	neighborhood.
	
\end{theorem}
\begin{proof}
	\par If $\tau> t$ we have that
	\begin{equation}
	\varphi(\tau)=L_f(\tau,t)\varphi(\tau)+\int_{t}^{\tau}L_f(t,s)
	[f(s,\varphi(s)) - D_xf(s,\xi(s)) \varphi(s)] ds.
	\end{equation}
	Thus, applying $Q(\tau)$ in the previous equation we obtain
	\begin{equation}
	Q(\tau)\varphi(\tau)=L_f(\tau,t)Q(t)\varphi(t)+\int_{t}^{\tau}L_f(t,s)
	Q(s)[f(s,\varphi(s)) - D_xf(s,\xi(s)) \varphi(s)] ds.
	\end{equation}
	Now, use that $L_f(\tau, t)|_{R(Q(t))}$ is invertible with inverse 
	$L_f(t,\tau)$ so we obtain
	\begin{equation}
	L_f(t,\tau)Q(\tau)\varphi(\tau)=Q(t)\varphi(t)+\int_{t}^{\tau}L_f(t,s)
	Q(s)[f(s,\varphi(s)) - D_xf(s,\xi(s)) \varphi(s)] ds.
	\end{equation}
	By definition we have that for $\tau>t$ suitable large
	\begin{equation*}
	\|L_f(t,\tau)Q(\tau)\varphi(\tau)\|\leq D e^{\nu |\tau|} e^{\alpha(t-\tau)} 
	\sup_{s\in \mathbb{R}}\|\varphi(s)\|\to 0, \hbox{ as } \tau\to +\infty.
	\end{equation*}
	Then
	\begin{equation}
	Q(t)\varphi(t)=-\int_{t}^{+\infty}L_f(t,s)
	Q(s)[f(s,\varphi(s)) - D_xf(s,\xi(s)) \varphi(s)] ds.
	\end{equation}
	Similarly, for $t>\tau$, as 
	\begin{equation*}
	\|L_f(t,\tau)(Id_X-Q(\tau))\varphi(\tau)\|\leq  D e^{\nu |\tau|} e^{-\alpha(t-\tau)} 
	\sup_{s\in \mathbb{R}}\|\varphi(s)\| \to 0, \hbox{ as } \tau\to -\infty,
	\end{equation*}
	thus
	\begin{equation}
	(Id_X-Q(t))\varphi(t)=\int_{-\infty}^{t}L_f(t,s)
	Q(s)[f(s,\varphi(s)) - D_xf(s,\xi(s)) \varphi(s)] ds.
	\end{equation}
	Therefore, the result follows by writing
	$\varphi(t)=(Id_X-Q(t))\varphi(t)+Q(t)\varphi(t)$ and using the previous expressions.
	\par Finally, if $\varphi$ is a bounded solution of $\mathcal{S}_f$ with
	$\sup_{t\in \mathbb{R}}\|\varphi(t)-\xi(t)\|\leq \epsilon$, then
	\begin{equation*}
	\sup_{t\in \mathbb{R}} \|\varphi(t) -\xi(t)\| \leq 2 D\rho(\epsilon) \alpha^{-1} 
	\sup_{t\in \mathbb{R}} \|\varphi(t) -\xi(t)\|.
	\end{equation*}
	For $\epsilon>0$ such that
	$2 D\rho(\epsilon)\alpha^{-1}<1 $ we see that 
	$\varphi(t)=\xi(t)$ for all $t\in \mathbb{R}$.
\end{proof}

\begin{remark}
	The Green function $\mathcal{G}_f$ satisfies
	\begin{equation*}
	\|G_f(t,s)\|_{\mathcal{L}(X)}\leq De^{\nu|s|} e^{-\alpha|t-s|},
	\end{equation*}
	for all $t,s\in \mathbb{R}$.
\end{remark}

\par Now, as an application of Theorem 
\ref{th-roughness-continuous-TED} 
we prove a result on the \textit{persistence of nonuniform hyperbolic solutions}.

\begin{theorem}[Persistence of nonuniform hyperbolic solutions]
	\label{th-persistence}
	Let $f:\mathbb{R}\times X\rightarrow X$ continuous with first continuous derivatives, $\mathcal{T}$ a linear evolution processes and $\mathcal{S}_f$ be evolution process generated by $f$ and $\mathcal{T}$. Assume that
	\begin{enumerate}
		\item $\mathcal{T}$ satisfies
		 \begin{equation}\label{th-hypothesis-1}
		\sup_{ 0\leq t-s\leq 1}\{e^{-\nu|t|} \|T(t,s)\|_{\mathcal{L}(X)} \}<+\infty,
		\end{equation}
		\item there is a global nonuniform hyperbolic 
		solution $\xi$ for $\mathcal{S}_f$, i.e., 
		$\mathcal{L}_f$ admits a nonuniform exponential dichotomy with bound 
		 $K(s)=De^{\nu|s|}$, for all $s\in \mathbb{R}$, and exponent $\alpha>\nu$.
		\item  $\xi$ is bounded with $\sup_{t\in \mathbb{R}}\|\xi(t)\|\leq M$;
		\item $f$ satisfies Condition 
		\eqref{equation-condition-to-obtain-persisntence-on-f};
		\item the derivative of $f$ is small through $\xi$, i.e.,
		\begin{equation*}
		\sup_{s\in \mathbb{R}}\sup_{\|x\|\leq M}\{e^{\nu |s|} \|D_xf(s,x)\|_{\mathcal{L}(X)}\} <+\infty;
		\end{equation*}
		\item $g:\mathbb{R} \times X \rightarrow X$ is differentiable with continuous derivatives and such that
		\begin{equation}\label{{th-persistence-hypothesis-4}}
		\sup_{|x|\leq M}	\|f(t,x)-g(t,x)\|_X + \|D_xf(t,x)-D_xg(t,x)\|_{\mathcal{L}(X)}<e^{-3\nu |t|} \min\bigg\{\, \frac{\epsilon}{4D\alpha^{-1}},\delta\, \bigg\},
		\end{equation}
		where $\delta>0$ is the same of Theorem \ref{th-perturation-of-nununiform-exp-dcihotomy-for-differential-eq}.
	\end{enumerate}
	Then there exists a unique nonuniform hyperbolic solution $\psi$ for 
	$\mathcal{S}_g$ such that
	\begin{equation*}
	\sup_{t\in \mathbb{R}} \|\xi(t)-\psi(t)\|<\epsilon.
	\end{equation*}
	
\end{theorem}

\begin{proof}
	If $y$ is a global bounded solution for $\mathcal{S}_g$, then, as in Remark
	\ref{lemma-non-uniform-hperbolic-solutions-first-lemma},
	we have that
	\begin{equation}
	\begin{split}
	y(t)=L_f(t,s)y(s)+\int_{s}^{t}L_f(t,\tau)
	[g(\tau,y(\tau)) - D_xf(\tau,\xi(\tau)) y(\tau)] d\tau,\\
	\xi(t)=L_f(t,s)\xi(s)+\int_{s}^{t}L_f(t,\tau)
	[f(\tau,\xi(\tau)) - D_xf(\tau,\xi(\tau)) \xi(\tau)] d\tau.
	\end{split}
	\end{equation}
	Thus $\phi(t)=y(t)-\xi(t)$ satisfies the following integral equation 
	\begin{equation}\label{equation-5112}
	\phi(t)=L_f(t,s)\phi(s)+\int_{s}^{t}L_f(t,\tau)
	h(\tau,\phi(\tau)) d\tau,
	\end{equation}
	where
	$h(t,\phi(t))=g(t,\phi(t)+\xi(t)) -f(t,\xi(t))-D_xf(t,\xi(t))\phi(t)$.
	\par Then, by Theorem
	\ref{lemma-non-uniform-hperbolic-solutions-second},
	there exists a bounded solution of 
	\eqref{equation-5112} in 
	\begin{equation*}
	B_\epsilon:=\{\phi :\mathbb{R}\rightarrow X \, : \phi \hbox{ is continuous and } \sup_{t\in \mathbb{R}} \|\phi(t)\|<\epsilon\},
	\end{equation*}
	if and only if, the operator
	\begin{equation*}
	(\mathcal{F}\varphi)(t)=\int_{-\infty}^{+\infty} G_f(t,s)
	h(s,\varphi(s)) ds,
	\end{equation*}
	has a fixed point in the space $B_\epsilon$.
	\par Now, we use the fact that $\mathcal{L}_f$ admits a nonuniform exponential 
	dichotomy to show
	that $\mathcal{F}$ has a unique fixed 
	point in $B_\epsilon$ for suitable small $\epsilon>0$.
	In order to use the Banach fixed point Theorem, we have to prove that
	$\mathcal{F}$ is a contraction and that
	$\mathcal{F}B_\epsilon\subset B_\epsilon$.
	\par First, let $\phi \in B_\epsilon$ we have
	\begin{eqnarray*}
		\| (\mathcal{F} \phi)(t)\|_X &\leq& D \int_{-\infty}^{+\infty}
		e^{\nu |s|} e^{-\alpha|t-s|} \|h(s,\phi(s))\|_X ds\\
		&\leq&  2D\alpha^{-1} \sup_{t\in \mathbb{R}} e^{\nu|t|} \, \|g(t,\xi(t)+\phi(t))-f(t,\xi(t)+\phi(t))\|_X\\
		&+&2D\alpha^{-1} \epsilon \sup_{\|x\|\leq \epsilon}
		\sup_{t\in \mathbb{R}} \frac{
			e^{\nu|t|} \,\|f(t,\xi(t)+x)-f(t,\xi(t))-D_xf(t,\xi(t))x\|_X 	
		}{\|x\|_X} \\
		&\leq & \epsilon/2+2\alpha^{-1}D\rho(\epsilon) \epsilon.
	\end{eqnarray*}
	Thus, choosing $\epsilon>0$ such that
	$4\alpha^{-1}D\rho(\epsilon)<1$, we see that
	$\mathcal{F}\phi\in B_\epsilon$.
	Now, we show that $\mathcal{F}$ is a contraction. 
	In fact, with similar computations we are able to prove for
	$\phi_1,\phi_2\in B_\epsilon$ that
	\begin{equation*}
	\| (\mathcal{F} \phi_1)(t)-(\mathcal{F} \phi_2)(t)\|_X
	\leq \frac{1}{2} \sup_{t\in \mathbb{R}} \|\phi_1(t)-\phi_2(t)\|_X.
	\end{equation*}
	Therefore, there is a unique fixed point $\phi$ in $B_\epsilon$ and we obtain 
	$\psi=\phi+\xi$ a global solution of 
	$\mathcal{S}_g$.
	\par Finally, we prove that $\psi$ is a nonuniform hyperbolic solution, that means, the linear evolution process $\mathcal{L}_g:=\{L_g(t,s): t\geq s\}$ that satisfies
	\begin{equation*}
	L_g(t,\tau)=T(t,\tau)+\int_{\tau}^{t} T(t,s) 
	D_xg(s,\psi(s)) L_g(s,\tau) ds
	\end{equation*}
	admits a nonuniform exponential dichotomy.
	\par To that end, we show that 
	$\mathcal{L}_f$
	satisfies conditions of Theorem 
	\ref{th-perturation-of-nununiform-exp-dcihotomy-for-differential-eq}
	and we see $\mathcal{L}_g$ as a small perturbation of $\mathcal{L}_f$.
	Indeed, since $\mathcal{T}$ satisfies \eqref{th-hypothesis-1} and 
	\begin{equation*}
	L_f(t,s)=T(t,s)+\int_{s}^{t} T(t,\tau)D_xf(\tau,\xi(\tau)) L_f(\tau,s)d\tau,
	\end{equation*}
	from a Grownwall inequality and assumption (5) we see that 
	\begin{equation*}
	\sup_{ 0\leq t-\tau\leq 1} \{e^{-\nu|t|}\,\|L_f(t,\tau)\|_{\mathcal{L}(X)}\} <+\infty.
	\end{equation*}
	\par To complete the proof, note that
	\begin{equation*}
	L_g(t,\tau)=L_f(t,\tau)+\int_{\tau}^{t}
	L_f(t,s) [D_xf(s,\xi(s))-D_xg(s,\psi(s))] L_g(s,\tau)ds.
	\end{equation*}
	Now, define
	$B(s):=D_xf(s,\xi(s))-D_xg(s,\psi(s))$ for all $s\in \mathbb{R}$.
	Then, by hypothesis 
	\eqref{{th-persistence-hypothesis-4}},
	we use
	Theorem 
	\ref{th-perturation-of-nununiform-exp-dcihotomy-for-differential-eq}
	to conclude that $\psi$ is a nonuniform hyperbolic solution of $\mathcal{S}_g$.
\end{proof}

\section{Conclusions}

\par The method of discretization, results 
\ref{th-continuousTED-implies-discreteTED} and
\ref{th-discrete-dichotomy-implies-continuous-dichotomy}, 
allowed us to compare continuous and discrete dynamical systems that exhibit nonuniform hyperbolicity. This approach was already known in the case of uniform exponential dichotomies, see for example
\cite{Chow-Leiva-1,Henry-1}, and in this work we established it for the nonuniform case.
Through this procedure we obtain:
\begin{enumerate}
	\item Uniqueness of the family of projections: Corollary \ref{cor-uniqueness-projection-continuous}.
	\item Continuous dependence of projections: Theorem \ref{th-continuous-depende-projections}.
	\item Robustness of nonuniform exponential dichotomies: theorems \ref{th-roughness-continuous-TED} and 
	\ref{th-perturation-of-nununiform-exp-dcihotomy-for-differential-eq}.
\end{enumerate} 
\par A disadvantage of the discretization method is that it is not possible to consider 
nonlinear growth rates $\rho(t)$, as in Barreira and Valls \cite{Barreira-Valls-Robustness-noninvertible}. 
On the other hand, it was possible to prove the robustness result with the assumption $\alpha>\nu$, which seems to be the sharpest one. In fact, if the growth of the bound $De^{\nu|s|}$ is greater or equal than the exponent $\alpha>0$ we do not know if the robustness results hold true.
\par The continuous dependence of projections and the persistence of hyperbolic solutions play an important role in the study of continuity of local unstable and stable manifolds for an associated nonlinear evolution process. In \cite{Carvalho-Langa} they use these permanence results to conclude continuity of pullback attractors under perturbation. On the other hand, it is not clear yet how to apply the results of stability of nonuniform hyperbolicity in the theory of attractors. However the persistence of nonuniform hyperbolic solutions and continuous dependence of projections should be important to study continuity of invariant manifolds associated to the nonuniform hyperbolic solutions. This, in turn, will be crucial in a possible application in the theory of attractors.

\section*{Acknowledgments}
\par This work was carried out while Alexandre Oliveira-Sousa visited the Dpto. Ecuaciones Diferenciales y An\'alisis Num\'erico (EDAN), Universidad de Sevilla and he wants to acknowledge the warm reception from people of EDAN.
We acknowledge the financial support from the following institutions: T. Caraballo and J. A. Langa by Ministerio de Ciencia, Innovaci\'on y Universidades (Spain), FEDER (European Community) under grant PGC2018-096540-B-I00, and by Proyecto I+D+i Programa Operativo FEDER Andaluc\'{\i}a US-1254251; A. N. Carvalho by S\~ao Paulo Research Foundation (FAPESP) grant 2018/10997-6, CNPq grant 306213/2019-2, and FEDER - Andalucía P18-FR-4509; and A. Oliveira-Sousa by S\~ao Paulo Research Foundation (FAPESP) grants 2017/21729-0 and 2018/10633-4, and CAPES grant PROEX-9430931/D.

\bibliographystyle{abbrv}
\bibliography{Bibliografia3}

\end{document}